\documentclass[11pt]{amsart}

\usepackage{amsmath,amssymb,amsthm}
\usepackage{bbm}
\usepackage{upgreek}
\usepackage{amsaddr}
\usepackage{mathtools}
\usepackage{geometry}
\usepackage{graphicx}
\usepackage{placeins}
\usepackage{needspace}
\usepackage{booktabs}
\usepackage{microtype}
\usepackage{hyperref}
\usepackage{enumerate}
\usepackage{pifont}
\usepackage{MnSymbol}
\usepackage{xcolor}

\graphicspath{{figures/}}
\hypersetup{
  colorlinks=true,
  linkcolor=blue!45!black,
  citecolor=green!35!black,
  urlcolor=blue!55!black
}

\newtheorem{theorem}{Theorem}[section]
\newtheorem{proposition}[theorem]{Proposition}
\newtheorem{lemma}[theorem]{Lemma}

\theoremstyle{definition}
\newtheorem{definition}[theorem]{Definition}
\newtheorem{example}[theorem]{Example}
\newtheorem{conjecture}[theorem]{Conjecture}
\theoremstyle{remark}
\newtheorem{remark}[theorem]{Remark}

\newcommand{\R}{\mathbb{R}}
\newcommand{\RP}{\mathbb{RP}}
\newcommand{\ART}{\operatorname{ART}}
\newcommand{\CSS}{\operatorname{CSS}}
\newcommand{\CPPOS}{\operatorname{CPPOS}}
\newcommand{\AreaEvolute}{\mathrm E_{0.5}}
\newcommand{\Int}{\operatorname{int}}
\newcommand{\mesh}{\operatorname{mesh}}
\newcommand{\dd}{\,\mathrm{d}}
\newcommand{\eps}{\varepsilon}

\newcommand{\doi}[1]{\href{https://doi.org/#1}{\nolinkurl{doi:#1}}}

\title[ARTs of CPPOSes]{Affine reflection transforms of convex polygons with parallel opposite sides}
\author{Peter Giblin$^{1}$, Stanis\l{}aw Janeczko$^{2}$, and Micha\l{} Zwierzy\'nski$^{2,\pentagram}$}

\thanks{E-mail: \href{mailto:pjgiblin@liverpool.ac.uk}{\nolinkurl{PJGiblin@liverpool.ac.uk}}}
\thanks{E-mail: \href{mailto:Stanislaw.Janeczko@pw.edu.pl}{\nolinkurl{Stanislaw.Janeczko@pw.edu.pl}}. ORCID: \href{https://orcid.org/0000-0002-8529-4647}{0000-0002-8529-4647}.}
\thanks{E-mail: \href{mailto:Michal.Zwierzynski@pw.edu.pl}{\nolinkurl{Michal.Zwierzynski@pw.edu.pl}}. ORCID: \href{https://orcid.org/0000-0002-9627-1563}{0000-0002-9627-1563}.}
\thanks{$^{\pentagram}$ Corresponding author.}

\address{$^{1}$The~University of Liverpool\\ Department of Mathematical Sciences\\ Liverpool L69 7ZL, United Kingdom}

\address{$^{2}$Warsaw University of Technology\\
Faculty of Mathematics and Information Science\\
ul. Koszykowa 75\\
00-662 Warsaw, Poland}

\keywords{affine reflection transform,
convex polygon, parallel opposite sides, conic envelope, centre symmetry set, Wigner caustic, central asymmetry}

\subjclass[2020]{Primary 52A10; Secondary 52A38, 52A40, 51N15}

\begin{document}

\raggedbottom

\begin{abstract}
We study affine reflection transforms of convex polygons with parallel opposite sides. An~opposite-side involution sends a~pencil through an~interior point to a~piecewise smooth family of lines. Its completed envelope consists of conic arcs and transition segments. The~completed transform has non-positive oriented area. Maximising the~negative oriented area and dividing by the~polygonal area gives an~affine invariant that vanishes precisely for centrally symmetric polygons. Combining an~identity for reflected chord vectors with the~Rogers--Shephard inequality shows that this invariant is strictly less than~$1$, with optimal universal constant~$1$. The~energy extends continuously to the~boundary and vanishes there. We also compare chord counts under compatible smoothing and prove Hausdorff convergence to the~smooth transform.
\end{abstract}

\maketitle

\setcounter{tocdepth}{1}
\tableofcontents

\section{Introduction}

\noindent A~convex polygon with parallel opposite sides has a~boundary involution: corresponding points on opposite sides are paired when they divide the~parallel sides in the~same proportion. Such polygons, abbreviated as $\CPPOS$, provide discrete counterparts of several constructions based on parallel tangents to a~smooth convex curve. Their centre symmetry sets, Wigner caustics, and related area inequalities have been studied in \cite{CraizerTeixeiraSilva,KonicerEtAl}.

For a~smooth oval, the~Wigner caustic is the~locus of midpoints of chords joining distinct points with parallel tangent lines. Its origin lies in Berry's semiclassical study of the~Wigner function \cite{BerryWigner}. The~centre symmetry set (CSS) is the~envelope of the~same chord lines. It was introduced through a~generalisation of the~centre of symmetry by Janeczko \cite{JaneczkoCSS} and developed geometrically by Giblin and Holtom \cite{GiblinHoltomCSS}. Both constructions are affine covariant and reduce to the~centre for a~centrally symmetric oval. For a~generic smooth oval, each has an~odd number of ordinary cusps, at least three, and the~CSS has at least as many cusps as the~Wigner caustic \cite[Theorem~7.2]{DomitrzRiosGCS}.

Dividing parallel-tangent chords in a~fixed ratio gives the~affine equidistants. Their singular points, as the~ratio varies, trace the~CSS. The~singularity theory of these constructions extends to hypersurfaces \cite{GiblinZakalyukinCSS}, to affine equidistants of higher-dimensional submanifolds \cite{DomitrzRiosRuasEquidistants}, and to Wigner caustics of Lagrangian submanifolds \cite{DomitrzManoelRiosShell}. Giblin and Reeve \cite{GiblinReeveCSS} study how the~CSS and equidistants change in families of plane curves. In the~plane, curvature criteria describe singular points \cite{DomitrzZwierzynskiSingular}, while decomposition into parallel arcs describes the~global branches, including those arising from non-convex curves \cite{DomitrzZwierzynskiGeometry}.

A~convex-geometric extension is Schneider's middle hedgehog \cite{SchneiderMiddle}, formed from midpoints of affine diameters. It agrees with the~Wigner caustic for a~smooth strictly convex body. For polygons, the~choice of pairing points on opposite sides matters, and we use the~equal-affine-parameter convention of \cite{CraizerTeixeiraSilva,KonicerEtAl}.

These loci also enter area inequalities. In particular, the~absolute oriented area of the~Wigner caustic gives a~sharp correction to the~classical isoperimetric inequality \cite{CufiGallegoReventos,ZhangMixedDeficit,ZwierzynskiImproved,ZwierzynskiRosettes,ZwierzynskiEqualities}.

\begin{figure}[h]
  \centering
  \includegraphics[width=\linewidth]{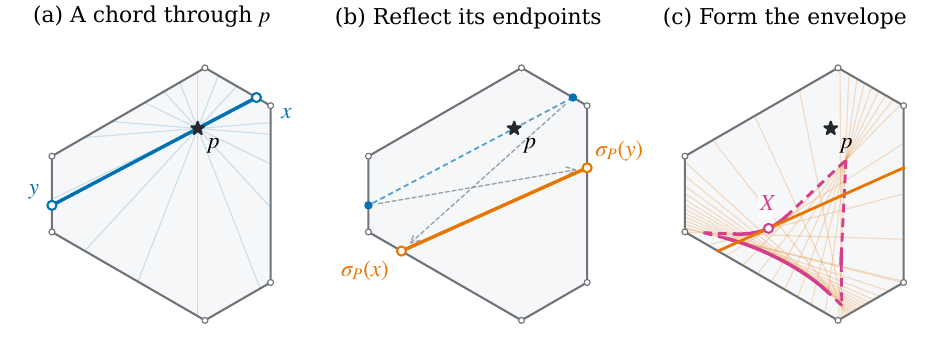}
\caption{Construction of $\ART_P(p)$ for $P=P_{1/10}$ from Example~\ref{prop:unit-energy-cppos-family} and $p=(1/10,3/10)$. (a) A~chord through $p$. (b) Its endpoints $x,y$ are paired with $\sigma_P(x),\sigma_P(y)$ at the~same affine parameter on opposite sides. (c) The~completed envelope has solid conic arcs and dashed transition segments. The~reflected line is tangent at $X$}
  \label{fig:art-construction-intro}
\end{figure}

We apply this boundary involution to the~two endpoints of every chord through a~fixed interior point $p$. The~joining lines of the~reflected endpoints form a~continuous, piecewise smooth family. Its envelope, completed by a~segment at each change of active sides, is the~\emph{affine reflection transform} $\ART_P(p)$. On each interval for which the~chord endpoints remain on fixed sides, the~envelope is an~arc of an~ellipse or a~hyperbola, or a~point. Figure~\ref{fig:art-construction-intro} illustrates the~endpoint pairing and the~formation of the~completed envelope.

This is the~polygonal counterpart of the~affine reflection transform in \cite{GJZARTSmooth}, which supplies the~general projective and support-function arguments used below. Here we compute the~envelope for each pair of active sides and analyse its completion and boundary limits. Section~\ref{sec:polygon-smooth-approximation} connects the~two constructions by compatible smoothing and tangent approximation.

For measures of symmetry and affine-covariant centres of convex bodies, see \cite{GrunbaumSymmetry,MeyerSchuttWerner}. We measure the~failure of reflected pencils to remain concurrent by maximising the~negative oriented area of the~ART over the~basepoint and dividing by the~area of the~polygon.

We prove the~following sharp area bound. Here $A$ and $A^*$ denote ordinary and oriented area, respectively.

\begin{theorem}\label{thm:intro-area-bound}
For every $\CPPOS$ $P$ and every $p\in\Int(P)$,
\begin{align*}
  0\leqslant -A^*\bigl(\ART_P(p)\bigr)<A(P).
\end{align*}
The~constant~$1$ in the~upper bound is optimal over all such polygons and basepoints.
\end{theorem}

The~sign follows from Theorem~\ref{thm:polygonal-art-area}. The~upper bound and its sharpness are proved in Theorem~\ref{thm:sharp-art-energy-bound} and Example~\ref{prop:unit-energy-cppos-family}.

The~main results are as follows.
\begin{enumerate}[(i)]
\item The~reflected-line map is an~affine-covariant involution. Its fixed lines recover the~polygonal centre symmetry set, its paired midpoints recover the~Wigner caustic, and every completed transform lies in the~original polygon (Theorem~\ref{thm:discrete-properties}).
\item Every cell has an~explicit rational parametrisation. An~exact incidence criterion distinguishes elliptic and hyperbolic cells from point degenerations. Adjacent-side cells and opposite-side cells always give points, with the~latter governed by a~projective involution (Theorem~\ref{thm:exact-cellwise} and Propositions~\ref{prop:cellwise-affine-type}, \ref{prop:adjacent-pencil}, and \ref{prop:opposite-pencil}).
\item The~oriented area has a~quadratic support-function formula. Its negative extends continuously to the~polygon and vanishes on the~boundary, even when the~completed front tends to a~doubled segment. The~normalised maximum detects central symmetry (Theorems~\ref{thm:polygonal-art-area} and \ref{thm:boundary-degeneration}, and Proposition~\ref{prop:art-invariant-properties}).
\item Compatible smoothings preserve the~boundary involution in the~limit. They yield a~generic comparison between centred chords and fixed chords. Conversely, the~ARTs of refining paired tangent polygons converge in the~Hausdorff metric to the~smooth ART (Proposition~\ref{prop:discrete-chambers} and Theorem~\ref{thm:hausdorff}).
\item The~area bound above follows from an~identity involving reflected chord vectors and the~Rogers--Shephard inequality. A~degenerating family of hexagons proves sharpness (Theorem~\ref{thm:sharp-art-energy-bound} and Example~\ref{prop:unit-energy-cppos-family}).
\end{enumerate}

\section{Projective and support-function notation}
\label{sec:polygon-preliminaries}

\noindent For $u,v\in\R^2$, write $[u,v]=\det(u,v)$. We regard a~convex polygon $P$ as the~closed convex body bounded by its sides and write $\Int(P)$ for its interior. For two distinct points $x,y$, the~notation $\overline{xy}$ denotes the~line through $x$ and $y$, whereas $\operatorname{conv}\{x,y\}$ denotes the~closed segment.

An~unoriented affine line $ax+by+c=0$, with $(a,b)\neq(0,0)$, is represented by $[a:b:c]$ in the~affine line space $\mathcal L=(\RP^2)^*\setminus\{[0:0:1]\}$. The~pencil through $p$ is
\begin{align*}
  \Lambda_p=\{\ell\in\mathcal L:p\in\ell\}\cong\RP^1.
\end{align*}
The~lines meeting $\Int(P)$ form an~open Möbius band. Indeed, projection onto the~unoriented normal direction has an~open interval of offsets as each fibre.

Let $t\mapsto\ell(t)$ be a~regular curve of lines and let $\boldsymbol\ell(t)\in\R^3\setminus\{0\}$ be a~homogeneous lift. The~envelope point is represented by
\begin{align}\label{eq:dual-envelope}
  \boldsymbol X(t)=\boldsymbol\ell(t)\times\boldsymbol\ell'(t).
\end{align}
At a~corner of the~curve of lines we join its two one-sided envelope points by the~segment contained in the~common generating line. This is the~completion used throughout the~paper.

We also use
\begin{align*}
  \mathbbm{n}(\theta)=(\cos\theta,\sin\theta),
  \qquad
  \mathbbm{t}(\theta)=(-\sin\theta,\cos\theta).
\end{align*}
If a~family of lines is written as $\mathbbm{n}(\theta)\cdot(z-p)=q(\theta)$, then its envelope on a~smooth parameter interval is
\begin{align*}
  X(\theta)=p+q(\theta)\mathbbm{n}(\theta)+q'(\theta)\mathbbm{t}(\theta).
\end{align*}
A~family parametrised once by unoriented normal directions has $q(\theta+\uppi)=-q(\theta)$. For a~closed piecewise smooth curve $X$, including its completing segments, its oriented area is
\begin{align*}
  A^*(X)=\frac{1}{2}\int[X,\dd X].
\end{align*}
We use the~orientation induced by increasing normal angle. The~ordinary area of a~convex body is denoted by $A$.

\section{The~polygonal affine reflection transform}

\subsection{The~boundary involution}

\noindent Following \cite{CraizerTeixeiraSilva,KonicerEtAl}, a~$\CPPOS$ is a~positively oriented convex polygon $P=(P_0,P_1,\ldots,P_{2n-1})$, $n\geqslant2$, whose opposite sides are parallel. Indices are taken modulo $2n$, and every listed point is assumed to be a~genuine vertex. With $e_i=P_{i+1}-P_i$, write
\begin{align*}
  e_{i+n}=-\mu_i e_i,
  \qquad
  \mu_i>0,
  \qquad
  \mu_{i+n}=\mu_i^{-1}.
\end{align*}
We set $L_i=P_i+\R e_i$ for the~supporting line of the~$i$th side.

\begin{definition}\label{def:sigma-p}
Define $\sigma_P:\partial P\to\partial P$ edgewise by
\begin{align}\label{eq:sigma-p}
  \sigma_P(P_i+se_i)=P_{i+n}+se_{i+n},
  \qquad 0\leqslant s\leqslant1.
\end{align}
\end{definition}

This defines an~orientation-preserving, fixed-point-free, piecewise affine involution, since the~formulas agree at vertices. Its restriction to each side has the~affine extension
\begin{align*}
  \widetilde\sigma_i:L_i\to L_{i+n},
  \qquad
  \widetilde\sigma_i(P_i+se_i)=P_{i+n}+se_{i+n},
  \qquad s\in\R.
\end{align*}

For a~secant line $\ell$ with boundary endpoints $x,y$, define
\begin{align*}
  \Phi_P(\ell)
  =
  \overline{\sigma_P(x)\,\sigma_P(y)}.
\end{align*}
On each \emph{cell}, where the~two endpoints remain in fixed open sides, $\Phi_P$ is smooth.

\begin{definition}\label{def:discrete-art}
For $p\in\Int(P)$, the~\emph{discrete affine reflection transform} $\ART_P(p)$ is the~completed envelope of the~reflected pencil $\Gamma^*_{P,p}=\Phi_P(\Lambda_p)$, with the~segment completion defined in Section~\ref{sec:polygon-preliminaries}.
\end{definition}

Equivalently, the~completed envelope consists of all limits of intersections $\Phi_P(\ell_k)\cap\Phi_P(\widetilde\ell_k)$ for distinct lines $\ell_k,\widetilde\ell_k\in\Lambda_p$ tending to $\ell\in\Lambda_p$, with the~possible ratios of the~two one-sided approaches included. 

\subsection{CSS, Wigner caustic, covariance, and containment}

\noindent Let $d_i=\overline{P_iP_{i+n}}$ be the~great diagonals and let $D_i=d_i\cap d_{i+1}$. The~\textit{centre symmetry set} of the~$\CPPOS$ is the~closed polygonal chain with vertices $D_0,\ldots,D_{n-1}$.  Put $M_i=\frac{1}{2}(P_i+P_{i+n})$. The~polygonal \textit{Wigner caustic} is the~closed polygonal chain with vertices $M_0,\ldots,M_{n-1}$. Since $D_{i+n}=D_i$ and $M_{i+n}=M_i$, both chains are traversed once in all oriented-area formulas below. The~discrete Wigner caustic and the~centre symmetry set are studied in \cite{CraizerTeixeiraSilva,KonicerEtAl}.

\Needspace{10\baselineskip}
\begin{theorem}\label{thm:discrete-properties}
The~discrete construction has the~following properties.

\begin{enumerate}[(a)]
\item $\Phi_P$ is an~involutive homeomorphism of the~secant-line Möbius band and is affine covariant.
\item The~fixed lines of $\Phi_P$ are precisely $\overline{x\,\sigma_P(x)}$, $x\in\partial P$.  Their completed projective dual is exactly $\CSS(P)$.
\item The~midpoint locus $\{(x+\sigma_P(x))/2:x\in\partial P\}$ is exactly $\AreaEvolute(P)$.
\item For every $p\in\Int(P)$, $\ART_P(p)\subset P$.
\item If $P$ is centrally symmetric with centre $o$, then $\sigma_P(x)=2o-x$ and $\ART_P(p)=\{2o-p\}$.
\end{enumerate}
\end{theorem}

\begin{proof}
Part (a) and the~fixed-line description in (b) follow from the~endpoint-involution argument of \cite{GJZARTSmooth}, with $\sigma_P$ in place of the~parallel-tangent involution. Here $\sigma_P$ is continuous and fixed-point-free, and its compatibility with affine maps follows directly from the~affine edge parameters in \eqref{eq:sigma-p}.

For $x(s)=P_i+se_i$, the~parallel-side homothety gives
\begin{align*}
  D_i=\frac{\mu_i P_i+P_{i+n}}{1+\mu_i}=\frac{\mu_i x(s)+\sigma_P(x(s))}{1+\mu_i}
\end{align*}
for every $s$. Thus the~fixed lines on this edge form a~pencil through $D_i$. At $P_i$, completion joins $D_{i-1}$ to $D_i$ along $d_i$, proving the~CSS assertion in (b).

Part (c) follows from $(x(s)+\sigma_P(x(s)))/2=(1-s)M_i+sM_{i+1}$.

For (d), apply the~cyclic-order containment argument of \cite{GJZARTSmooth} to the~orientation-preserving boundary involution $\sigma_P$. All one-sided envelope limits belong to $P$, and convexity also places their completing segments in $P$.

If $P$ is centrally symmetric, \eqref{eq:sigma-p} becomes $\sigma_P(x)=2o-x$, and (e) follows.

\end{proof}

\begin{figure}[ht]
  \centering
  \includegraphics[width=.44\linewidth]{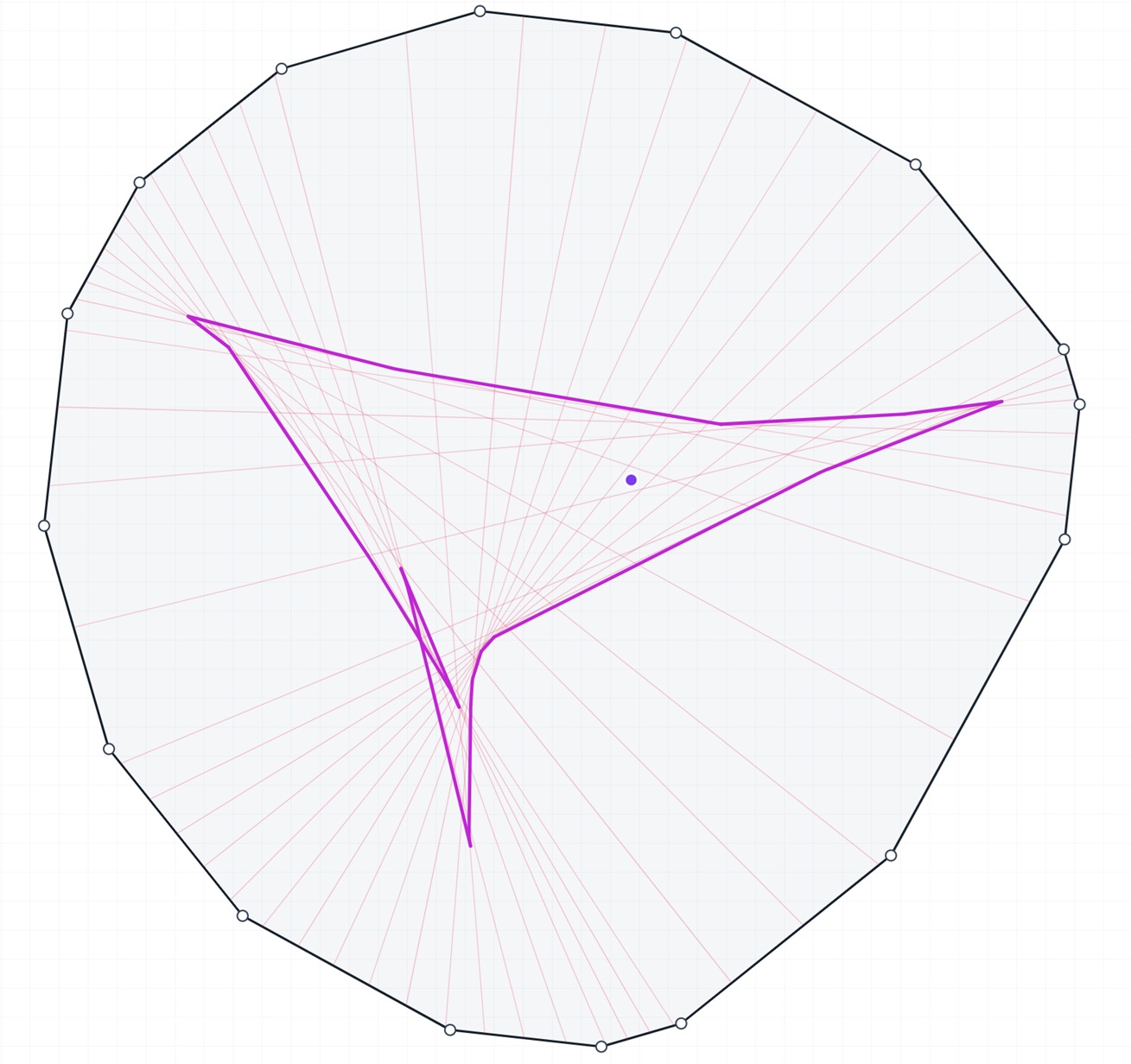}
  \includegraphics[width=.44\linewidth]{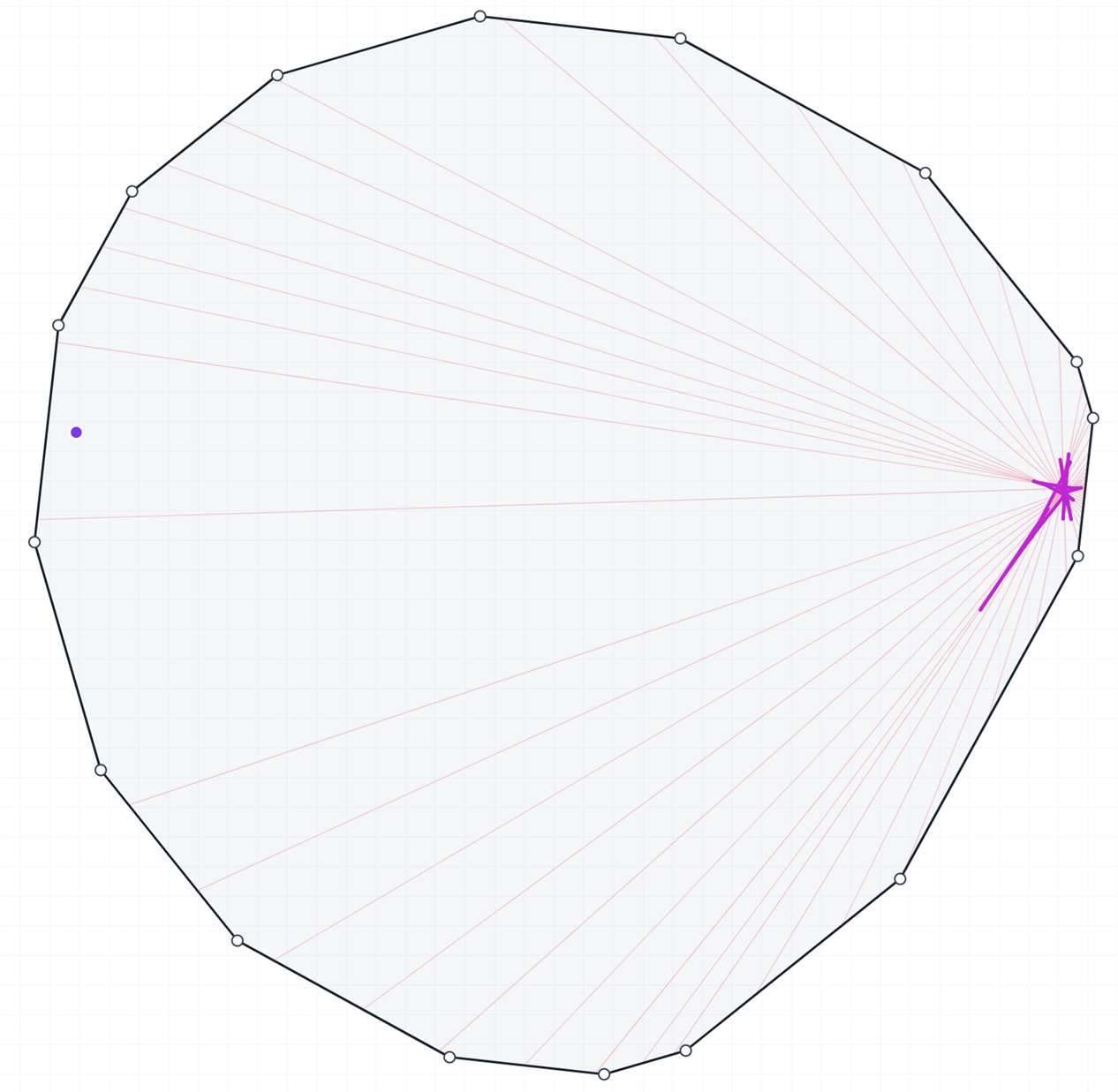}
  \includegraphics[width=.44\linewidth]{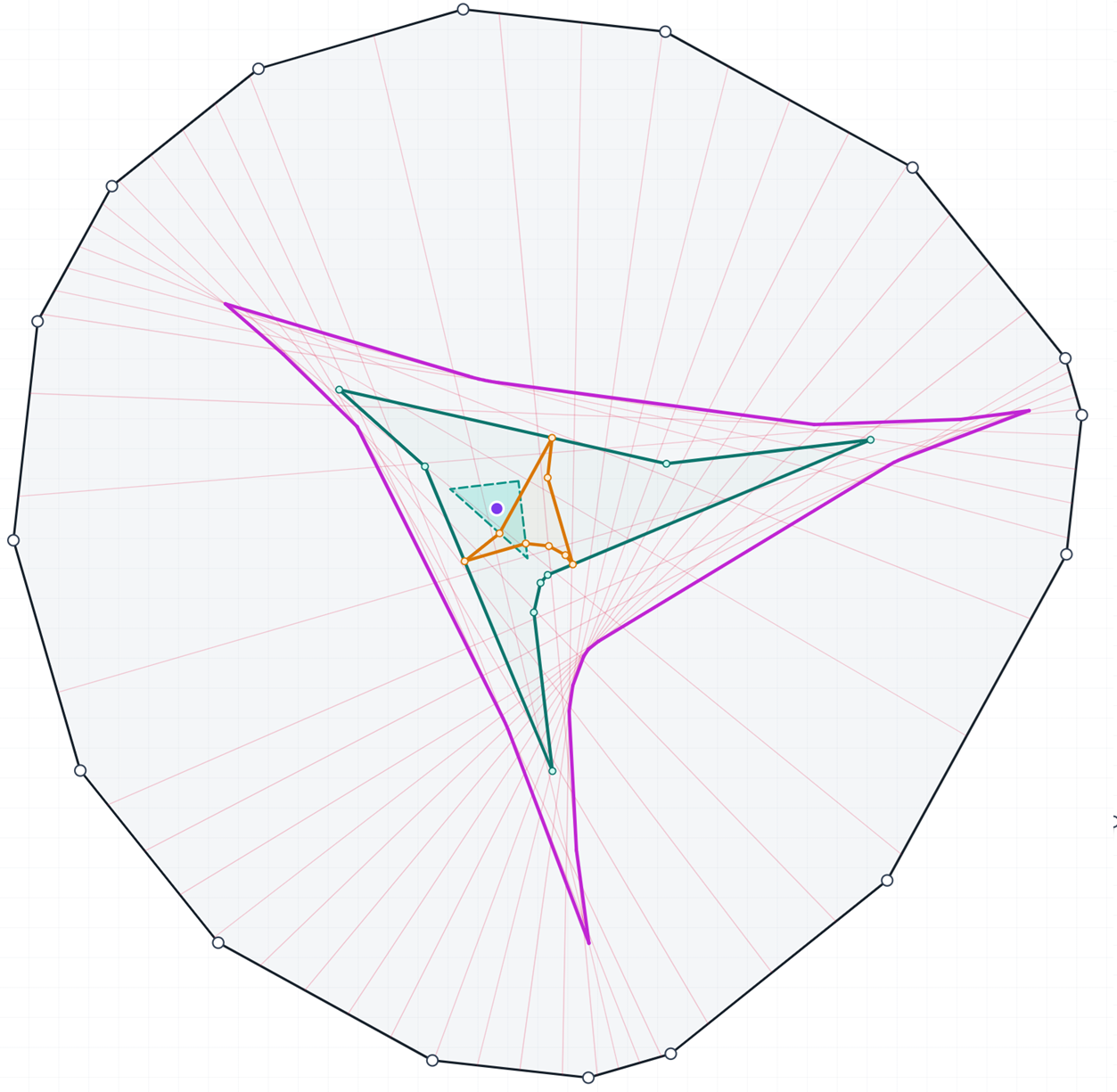}
  \includegraphics[width=.44\linewidth]{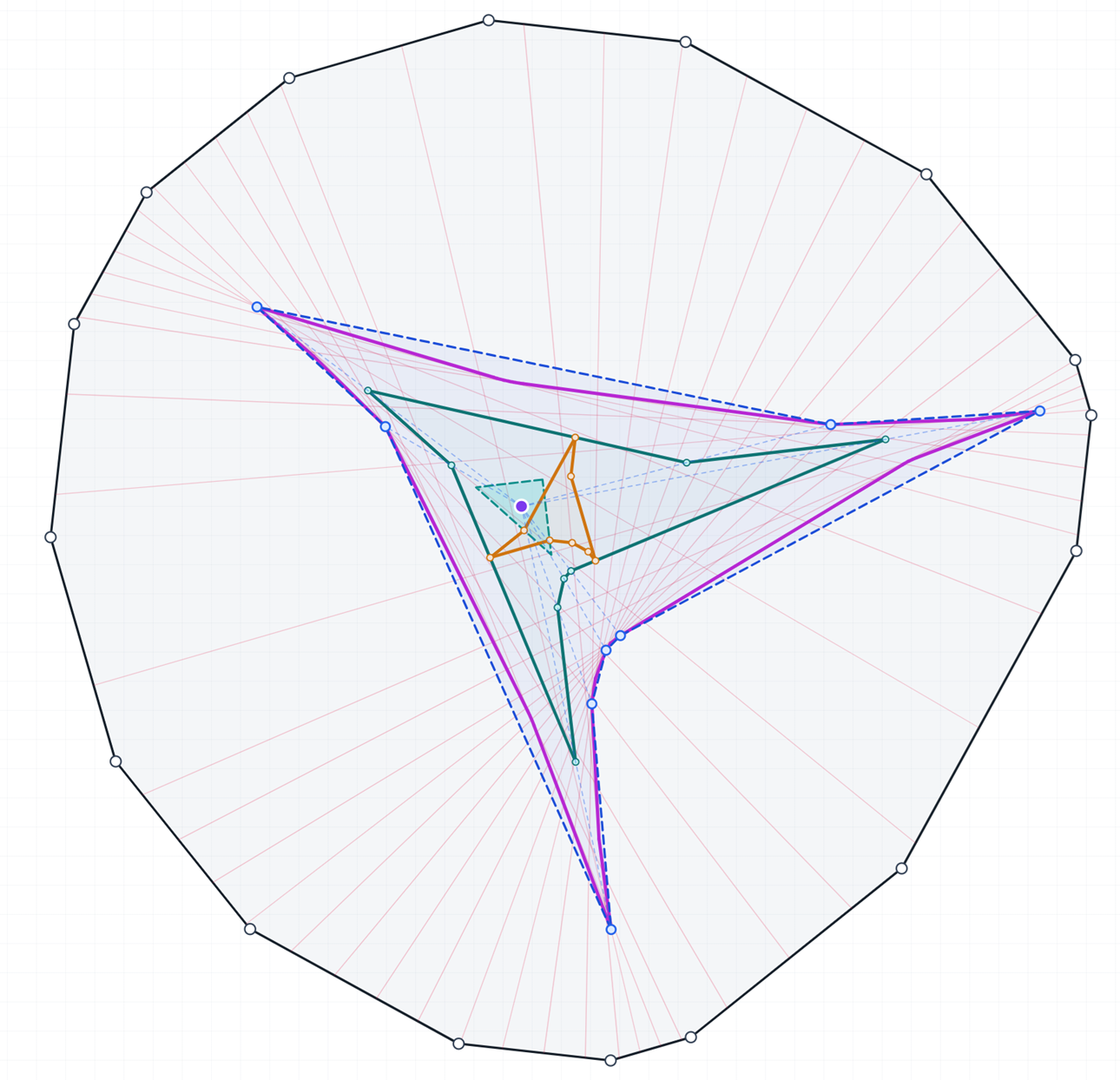}
\caption{Top: $16$-gons with chosen basepoints and their ARTs. Bottom: basepoints in the~kernel of the~centre symmetry set, with the~Wigner caustic, centre symmetry set, and discrete ART.  Bottom right: the~points $Q_i(p)$ and the~polygon $\mathcal Q$ they determine}
  \label{fig:discrete-example}
\end{figure}

\section{Exact cellwise formulas and affine conic types}

\noindent We give explicit formulas for the~$\ART$ on each cell of the~pencil. For affine points and vectors, use the~homogeneous lifts $\widehat x=(x,1)$ and $\widehat v=(v,0)$.

\begin{theorem}\label{thm:exact-cellwise}
Fix a~cell of the~pencil through $p$ whose chord endpoints lie in the~open sides $e_i$ and $e_j$. Write them as $x(s)=P_i+se_i$ and $y(t)=P_j+te_j$. Their parameters satisfy
\begin{align}\label{eq:t-of-s}
  t(s)=-\frac{[x(s)-p,P_j-p]}{[x(s)-p,e_j]}.
\end{align}
A~homogeneous lift of the~reflected line is the~quadratic polynomial
\begin{align}\label{eq:quadratic-dual-cell}
  \boldsymbol\ell(s)
  &=\widehat{P_{i+n}+se_{i+n}}\times
  \Bigl([x(s)-p,e_j]\widehat{P_{j+n}}-[x(s)-p,P_j-p]\widehat{e_{j+n}}\Bigr)=\boldsymbol\ell_0+s\boldsymbol\ell_1+s^2\boldsymbol\ell_2,
\end{align}
where $\boldsymbol\ell_0$, $\boldsymbol\ell_1$, and $\boldsymbol\ell_2$ are its constant, linear, and quadratic coefficients. At every regular parameter, the~envelope point is represented by
\begin{align}\label{eq:exact-cell-envelope}
  \boldsymbol Z(s)
  =\boldsymbol\ell_0\times\boldsymbol\ell_1
  +2s\,\boldsymbol\ell_0\times\boldsymbol\ell_2
  +s^2\,\boldsymbol\ell_1\times\boldsymbol\ell_2.
\end{align}
Dividing by the~third coordinate gives a~rational parametrisation of a~conic or its degeneration. If no line through $p$ contains two vertices of $P$, the~pencil has exactly $2n$ cells, so the~completed $\ART$ has at most $2n$ conic arcs and $2n$ transition segments.
\end{theorem}

\begin{proof}
Solving the~collinearity equation $[P_i-p+se_i,P_j-p+te_j]=0$ gives \eqref{eq:t-of-s}. The~reflected endpoints are $P_{i+n}+se_{i+n}$ and $P_{j+n}+t(s)e_{j+n}$. Multiplying the~second homogeneous point by the~denominator in \eqref{eq:t-of-s} gives \eqref{eq:quadratic-dual-cell}. Expanding $\boldsymbol\ell(s)\times\boldsymbol\ell'(s)$ gives \eqref{eq:exact-cell-envelope} by \eqref{eq:dual-envelope}.

The~homogeneous quadratic parameter $[1:2s:s^2]$ parametrises a~projective conic (see \cite{HarrisAG}). The~linear map in \eqref{eq:exact-cell-envelope} gives a~conic if its rank is three, and a~line or point if its rank drops. Propositions~\ref{prop:cellwise-affine-type} and \ref{prop:opposite-pencil} below show that only point degenerations occur here.

Cell transitions occur at lines through vertices. Under the~stated assumption, their $2n$ directions are distinct in $\RP^1$ and divide the~pencil into $2n$ cells. Joining the~one-sided envelope points at each transition gives the~completion of Definition~\ref{def:discrete-art}.
\end{proof}

\begin{remark}\label{rem:art-not-polygon}
In general, the~completed polygonal $\ART$ is a~closed front of curved conic arcs and straight transition segments. It is an~actual polygon only when every active conic piece degenerates to a~point.
\end{remark}

\begin{proposition}
\label{prop:cellwise-affine-type}
Fix a~non-empty cell supported on $e_i$ and $e_j$, and suppose that $L_i$ and $L_j$ are not parallel.  Let $x,y:\R^2\to\R$ be the~affine functions uniquely determined by
\begin{align*}
  L_j&=\{x=0\},&L_{j+n}&=\{x=1\},
  &
  L_i&=\{y=0\},&L_{i+n}&=\{y=1\}.
\end{align*}
Write $p=(u,v)$ in these coordinates. Then $0<u,v<1$.  Put $S=L_i\cap L_j$ and $S^*=L_{i+n}\cap L_{j+n}$. In these coordinates $S^*=(1,1)$. Define $A,B\in\R$ by
\begin{align*}
  \widetilde\sigma_i(S)=(1+A,1),
  \qquad
  \widetilde\sigma_j(S)=(1,1+B).
\end{align*}
Finally, set
\begin{align}\label{eq:cell-beta}
  \beta_{ij}(p)
  =
  AB-A\mu_jv-B\mu_iu.
\end{align}
This quantity is unchanged if $i$ and $j$ are interchanged. The~cellwise envelope has exactly one of the~following affine types:
\begin{enumerate}[(a)]
\item If $\beta_{ij}(p)>0$, it is an~arc of a~non-degenerate ellipse.
\item If $\beta_{ij}(p)<0$, it is an~arc of a~non-degenerate hyperbola.
\item If $\beta_{ij}(p)=0$, all reflected lines in the~cell are concurrent at the~finite point with coordinates
\begin{align}\label{eq:cell-concurrency-point}
  (1+A-\mu_i u,\;1+B-\mu_j v),
\end{align}
and the~cellwise envelope is that point.
\end{enumerate}
In particular, a~parabolic cell and a~one-dimensional degenerate cellwise envelope cannot occur.  The~point-degeneration criterion has the~coordinate-free form
\begin{align}\label{eq:cell-point-geometric-criterion}
  [\widetilde\sigma_i^{-1}(S^*)-p,
  \widetilde\sigma_j^{-1}(S^*)-p]=0.
\end{align}

\end{proposition}

\begin{proof}
Use $(x,y)$ as affine coordinates and centre the~output at $S^*=(1,1)$, writing $X=x-1$ and $Y=y-1$. Original endpoints $(r,0)$ and $(0,t)$ are collinear with $p=(u,v)$ precisely when
\begin{align*}
  t=\frac{vr}{r-u}.
\end{align*}
Their reflected endpoints have centred coordinates $(\xi,0)=(A-\mu_i r,0)$ and $(0,\eta)=(0,B-\mu_j t)$. Eliminating $r,t$ gives
\begin{align}\label{eq:cell-projectivity-normal-form}
  \eta=\frac{(\mu_jv-B)\xi+\beta_{ij}(p)}{A-\mu_iu-\xi}.
\end{align}
The~determinant of this projectivity is
\begin{align}\label{eq:cell-projectivity-determinant}
  (\mu_jv-B)(A-\mu_iu)+\beta_{ij}(p)
  =-\mu_i\mu_juv<0.
\end{align}
Substituting \eqref{eq:cell-projectivity-normal-form} in the~joining-line equation $\eta X+\xi Y-\xi\eta=0$ gives
\begin{align}\label{eq:cell-tangent-quadratic}
 -(Y+\mu_jv-B)\xi^2
 +\bigl((\mu_jv-B)X+(A-\mu_iu)Y-\beta_{ij}(p)\bigr)\xi
 +\beta_{ij}(p)X=0.
\end{align}
Suppose first that $\beta_{ij}(p)\neq0$. The~double-root condition gives the~envelope conic
\begin{align*}
 \bigl((\mu_jv-B)X+(A-\mu_iu)Y-\beta_{ij}(p)\bigr)^2
 +4\beta_{ij}(p)X(Y+\mu_jv-B)=0.
\end{align*}
The~determinant of its homogeneous symmetric matrix is $-4\beta_{ij}(p)^2(\mu_i\mu_juv)^2\neq0$. Its quadratic part has discriminant $-16\beta_{ij}(p)\mu_i\mu_juv$, so the~conic is elliptic for $\beta_{ij}(p)>0$ and hyperbolic for $\beta_{ij}(p)<0$. It is real because the~cellwise parametrisation supplies real envelope points. A~parabola would require a~zero discriminant.

If $\beta_{ij}(p)=0$, cancel $\xi$ in \eqref{eq:cell-tangent-quadratic}, initially for $\xi\neq0$ and then by continuity. The~resulting lines all pass through $(X,Y)=(A-\mu_iu,B-\mu_jv)$, which gives \eqref{eq:cell-concurrency-point}. Their envelope is that point.

Finally, $\widetilde\sigma_i^{-1}(S^*)=(A/\mu_i,0)$ and $\widetilde\sigma_j^{-1}(S^*)=(0,B/\mu_j)$ in the~original coordinates. Their collinearity with $(u,v)$ is equivalent to $AB-A\mu_jv-B\mu_iu=0$, proving \eqref{eq:cell-point-geometric-criterion}. Note that interchanging $i$ and $j$ swaps $(x,u,A,\mu_i)$ with $(y,v,B,\mu_j)$ and leaves \eqref{eq:cell-beta} unchanged.
\end{proof}

\begin{proposition}
\label{prop:adjacent-pencil}
Suppose that a~cell is supported on the~adjacent sides $e_i$ and $e_{i+1}$.  Let $T_i:\R^2\to\R^2$ be the~unique affine map defined by
\begin{align*}
  T_i(P_{i+1}+r e_i+s e_{i+1})
  =
  P_{i+n+1}+r e_{i+n}+s e_{i+n+1}.
\end{align*}
Then every reflected line belonging to this cell passes through $T_i(p)$.  Consequently, the~corresponding conic piece degenerates to the~single point $T_i(p)$.
\end{proposition}

\begin{proof}
The~two consecutive edge directions are linearly independent, so $T_i$ is a~well-defined nonsingular affine map.  Its restrictions to $L_i$ and $L_{i+1}$ are precisely $\widetilde\sigma_i$ and $\widetilde\sigma_{i+1}$. Therefore, if a~line through $p$ has its endpoints on these two sides, the~reflected line is its image under $T_i$ and hence contains $T_i(p)$.
\end{proof}

\begin{remark}
The~adjacent-side case is automatic, but it is not the~only way a~nonparallel cell can degenerate.  For a~fixed pair $(i,j)$, put $U_i=\widetilde\sigma_i^{-1}(S^*)$ and $U_j=\widetilde\sigma_j^{-1}(S^*)$. If $U_i\neq U_j$, condition \eqref{eq:cell-point-geometric-criterion} says that, among the~basepoints for which the~$(i,j)$-cell is non-empty, the~exceptional ones are exactly those on the~fixed line $U_iU_j$.  If $U_i=U_j$, the~condition holds for every basepoint for which the~cell is non-empty.  The~latter occurs for adjacent sides and, for every side pair, when $P$ is centrally symmetric. In the~centrally symmetric case all cellwise points equal $2o-p$.
\end{remark}

\begin{remark}\label{rem:both-affine-types}
Take the~$\CPPOS$ with vertices
\begin{align*}
  &(0,0),(1,0),(2,1),(2,2),(1,3),(0,3),
  \left(-\frac{7}{5},\frac{8}{5}\right),
  \left(-\frac{7}{5},\frac{7}{5}\right)
\end{align*}
and the~interior basepoint $p=(3/25,9/20)$.  The~cells supported on $(e_2,e_7)$ and $(e_0,e_6)$ are both non-empty: chords through $p$ are obtained with the~respective pairs of open-side parameters $(1/5,209/263)$ and $(3/4,9/28)$.  The~data $(A,B,u,v;\mu_i,\mu_j)$ from Proposition~\ref{prop:cellwise-affine-type} are, respectively,
\begin{align*}
  \left(-\frac{4}{5},-\frac{2}{7},\frac{57}{400},\frac{47}{85};
  \frac{1}{5},\frac{5}{7}\right)
  \quad\text{and}\quad
  \left(\frac{2}{17},2,\frac{38}{85},\frac{3}{20};1,5\right).
\end{align*}
Thus \eqref{eq:cell-beta} gives
\begin{align*}
  \beta_{2,7}(p)=\frac{65769}{119000}>0,
  \qquad
  \beta_{0,6}(p)=-\frac{127}{170}<0.
\end{align*}
Thus the~first cell is elliptic and the~second is hyperbolic. Figure~\ref{fig:cellwise-affine-types} shows these two pencil cells and their envelope arcs.
\end{remark}

\begin{figure}[ht]
  \centering
  \includegraphics[width=.97\linewidth]{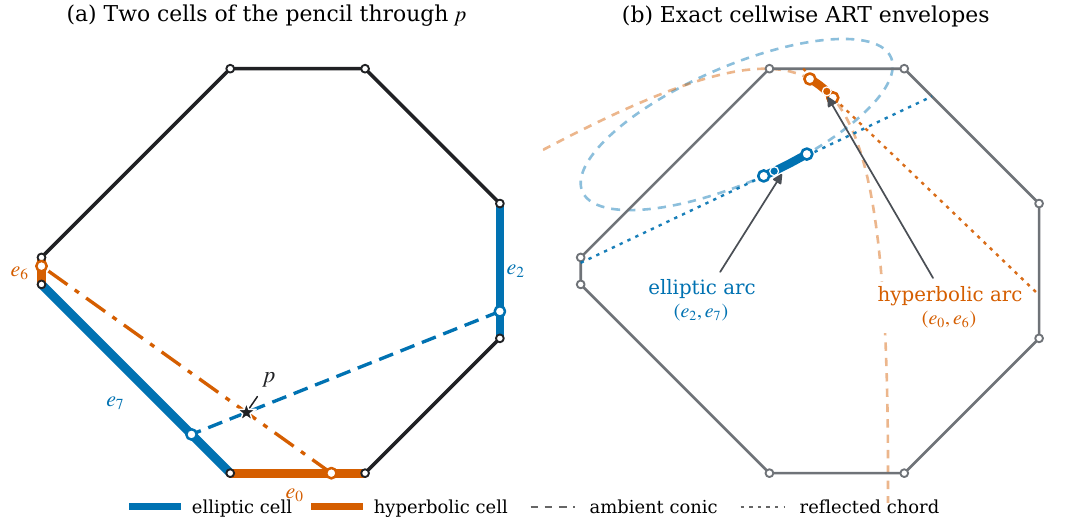}
\caption{The~two cells from Remark~\ref{rem:both-affine-types} and their elliptic and hyperbolic $\ART$ arcs.  Dashed continuations in the~right panel are the~ambient conics}
  \label{fig:cellwise-affine-types}
\end{figure}

For a~cell supported on opposite sides, the~envelope also degenerates to a~point.

\begin{proposition}\label{prop:opposite-pencil}
Suppose that the~cell in Theorem~\ref{thm:exact-cellwise} is supported on $e_i$ and $e_{i+n}$.  Then $D_i=(\mu_iP_i+P_{i+n})/(1+\mu_i)$. Let $\tau_i:\R^2\to\R$ be the~affine function equal to $0$ on the~line containing $e_i$ and to $1$ on the~line containing $e_{i+n}$.  Then all reflected lines belonging to this cell pass through the~single point
\begin{align}\label{eq:Qi-mu}
  Q_i(p)
  =
  D_i-
  \frac{\mu_i}
  {1+(\mu_i^2-1)\tau_i(p)}
  (p-D_i).
\end{align}
Thus the~corresponding conic piece degenerates to the~point $Q_i(p)$.

The~map $R_i:p\mapsto Q_i(p)$ extends to a~projective involution of $\RP^2$.  Every projective line through $D_i$ is invariant under $R_i$.  For $p\in\Int(P)$, the~point $D_i$ belongs to the~closed segment $\operatorname{conv}\{p,Q_i(p)\}$.
\end{proposition}

\begin{proof}
Apply an~affine change of coordinates taking the~two side-lines to $y=0$ and $y=1$.  We may write $P_i=(0,0)$, $e_i=(1,0)$, $P_{i+n}=(c,1)$, and $e_{i+n}=(-\mu_i,0)$. Then $\tau_i(x,y)=y$ and
\begin{align*}
  D_i=
  \left(\frac{c}{1+\mu_i},\frac{1}{1+\mu_i}\right).
\end{align*}
Let $p=(u,v)$, where $0<v<1$.  A~chord with endpoints $x(s)=(s,0)$ and $y(t)=(c-\mu_it,1)$ passes through $p$ precisely when
\begin{align*}
  t(s)
  =
  \frac{(1-v)s+vc-u}{\mu_iv}.
\end{align*}
Its reflected line joins $(t(s),0)$ to $(c-\mu_is,1)$.  Direct substitution shows that every such line contains
\begin{align}\label{eq:Q-coordinate}
  \left(
  \frac{c(1+(\mu_i-1)v)-\mu_iu}
  {1+(\mu_i^2-1)v},
  \frac{1-v}{1+(\mu_i^2-1)v}
  \right).
\end{align}
Rewriting \eqref{eq:Q-coordinate} relative to $D_i$ gives \eqref{eq:Qi-mu}.

The~numerator and denominator in \eqref{eq:Qi-mu} are affine-linear in $p$, so $R_i$ extends projectively and preserves every line through $D_i$. Moreover,
\begin{align*}
  \tau_i(Q_i(p))
  =\frac{1-\tau_i(p)}{1+(\mu_i^2-1)\tau_i(p)}.
\end{align*}
Substituting this expression in \eqref{eq:Qi-mu} gives $Q_i(Q_i(p))=p$, so $R_i$ is an~involution. Since $0<\tau_i(p)<1$, the~denominator in \eqref{eq:Qi-mu} is positive. Hence $Q_i(p)-D_i$ is a~negative multiple of $p-D_i$. This proves the~segment assertion, including $p=Q_i(p)=D_i$.
\end{proof}

\begin{remark}\label{rem:opposite-similarity}
If $\mu_i=1$, then $Q_i(p)=2D_i-p$, which is $2o-p$ for a~centrally symmetric polygon. The~collinearity of $p,D_i,Q_i(p)$ relates each active opposite-side cell to its $\CSS$ vertex. It does not give a~single affine map from the~entire $\CSS$ to the~$\ART$: non-opposite cells depend on the~positions and directions of all four side-lines in \eqref{eq:exact-cell-envelope}, not only on side lengths.

If $x$ lies in the~open side $e_i$ and $p=(x+\sigma_P(x))/2$, then $p$ lies on the~Wigner-caustic edge $\operatorname{conv}\{M_i,M_{i+1}\}$ and $\tau_i(p)=1/2$. Thus \eqref{eq:Qi-mu} becomes
\begin{align*}
  Q_i(p)=D_i-\frac{2\mu_i}{1+\mu_i^2}(p-D_i).
\end{align*}
The~fixed chord line $\overline{x\,\sigma_P(x)}$ is a~common generating line of the~$\ART$ and $\CSS$, with contacts $Q_i(p)$ and $D_i$. These contacts are shown in Figure~\ref{fig:discrete-example}.
\end{remark}

\begin{proposition}
\label{prop:Q-polygon-area}
Let $p\in\Int(P)$. The~cyclic polygon $\mathcal Q_P(p)$ has vertices
\begin{align*}
 Q_0(p),\ldots,Q_{n-1}(p).
\end{align*}
It is defined whether or not every opposite-side cell is active for the~pencil through $p$. Put
\begin{align*}
  r_i(p)
  =
  1+
  \frac{\mu_i}
  {1+(\mu_i^2-1)\tau_i(p)}.
\end{align*}
Then $r_i(p)>1$, $Q_i(p)-p=r_i(p)(D_i-p)$, and
\begin{align}\label{eq:Q-polygon-area}
  A^*\bigl(\mathcal Q_P(p)\bigr)
  =
  \frac{1}{2}\sum_{i=0}^{n-1}
  r_i(p)r_{i+1}(p)
  [D_i-p,D_{i+1}-p].
\end{align}
In particular, if $p$ belongs to the~oriented kernel of the~CSS, meaning that
\begin{align}\label{eq:oriented-CSS-kernel}
  [D_i-p,D_{i+1}-p]\leqslant0
  \quad\text{for every }i,
\end{align}
then
\begin{align}\label{eq:Q-CSS-E05-order}
  \left|A^*\bigl(\mathcal Q_P(p)\bigr)\right|
  \geqslant
  \left|A^*\bigl(\CSS(P)\bigr)\right|
  \geqslant
  \left|A^*\bigl(\AreaEvolute(P)\bigr)\right|.
\end{align}
\end{proposition}

\begin{proof}
Rearranging \eqref{eq:Qi-mu} gives $Q_i-p=r_i(D_i-p)$.  Translating the~area formula to the~basepoint $p$ now proves \eqref{eq:Q-polygon-area}.  Under \eqref{eq:oriented-CSS-kernel}, every summand has the~conventional non-positive orientation and $r_ir_{i+1}>1$.  This gives the~first inequality in \eqref{eq:Q-CSS-E05-order}.  The~second is the~polygonal CSS--Wigner caustic inequality in \cite[Theorem~4.7]{KonicerEtAl}. Its area argument also applies without the~angle restriction used in that paper. Indeed, put $v_i=\frac{1}{2}(P_i-P_{i+n})$ and $\eta_i=(\mu_i-1)/(\mu_i+1)$. Then $D_i=M_i+\eta_i v_i=M_{i+1}+\eta_i v_{i+1}$, and cyclic summation gives
\begin{align*}
 A^*(\CSS(P))
 &=A^*(\AreaEvolute(P))
 -\frac{1}{2}\sum_{i=0}^{n-1}\eta_i^2[v_i,v_{i+1}].
\end{align*}
Here $v_{i+n}=-v_i$ and $\eta_{i+n}=-\eta_i$. The~points $v_0,\ldots,v_{2n-1}$ are the~positively ordered vertices of the~central symmetral $V=\frac{1}{2}(P-P)$, so $[v_i,v_{i+1}]>0$. The~area decomposition $A(P)=2A^*(\AreaEvolute(P))+A(V)$, together with $A(V)\geqslant A(P)$ from the~Brunn--Minkowski inequality, gives $A^*(\AreaEvolute(P))\leqslant0$. Thus $A^*(\CSS(P))\leqslant A^*(\AreaEvolute(P))\leqslant0$, which proves the~required comparison for every $\CPPOS$ considered here.
\end{proof}

\begin{remark}
The~polygon $\mathcal Q_P(p)$ is an~auxiliary object, not the~full ART. Only the~points associated with active opposite-side cells belong to the~cellwise construction for a~given $p$, and the~remaining $\ART$ pieces are generally curved conic arcs with transition segments. See Figure~\ref{fig:discrete-example}.
\end{remark}

\section{Oriented area, boundary degeneration, and affine centres}

\begin{proposition}\label{prop:polygonal-direction}
For every $\CPPOS$ $P$ and $p\in\Int(P)$, the~direction map $\Gamma^*_{P,p}\to\RP^1$ is an~orientation-preserving homeomorphism, with positive derivative on each open cell. Thus the~normal direction parametrises the~transformed line family globally.
\end{proposition}

\begin{proof}
Apply the~endpoint angular-derivative calculation of \cite{GJZARTSmooth} on each open cell. The~reflected endpoints move positively because $\sigma_P$ preserves cyclic order, and convexity makes the~output angular derivative positive: the~reflected chord joins distinct open sides and is tangent to neither. Continuity at transitions prevents reversal.

The~direction projection makes the~secant-line Möbius band an~interval bundle over $\RP^1$. By Theorem~\ref{thm:discrete-properties}(a), $\Gamma^*_{P,p}$ is the~homeomorphic image of its core circle $\Lambda_p$. An~embedded one-sided circle in a~Möbius band is isotopic to the~core, so the~restricted projection has degree $\pm1$. Local positivity makes the~degree $+1$ and excludes multiple covering.
\end{proof}

Parametrise the~transformed unoriented lines by their normal angle $\theta\in\R/\uppi\mathbb Z$ and choose the~continuous lift
\begin{align}\label{eq:polygonal-line-support}
  \mathbbm{n}(\theta)\cdot(z-p)=q_p(\theta),
  \qquad
  q_p(\theta+\uppi)=-q_p(\theta).
\end{align}
The~function $q_p$ is continuous and piecewise smooth.  At a~cell transition its one-sided derivatives may differ, and the~completed envelope contains the~segment between the~two one-sided envelope points.

\begin{lemma}\label{lem:polygonal-support-continuity}
Let $p_\nu\in\Int(P)$ and $p_\nu\to p_0\in P$.  If $p_0\in\Int(P)$, let $q_{p_0}$ be defined by \eqref{eq:polygonal-line-support}.  If $p_0\in\partial P$, put $p_0^*=\sigma_P(p_0)$ and $q_{p_0}(\theta)=(p_0^*-p_0)\cdot \mathbbm{n}(\theta)$. Then
\begin{align}\label{eq:polygonal-support-H1}
  q_{p_\nu}\to q_{p_0}\quad\text{in }H^1(0,\uppi).
\end{align}
In particular, the~quadratic area form $\int_0^\uppi(q_p^2-(q_p')^2)\,\dd\theta$ is continuous in $p$.
\end{lemma}

\begin{proof}
On the~interior of every cell the~support-envelope formula gives
\begin{align*}
  X_{p_\nu}
  =
  p_\nu+q_{p_\nu}\mathbbm{n}+q_{p_\nu}'\mathbbm{t}.
\end{align*}
Theorem~\ref{thm:discrete-properties}(d) gives $X_{p_\nu}\in P$.  Consequently,
\begin{align}\label{eq:polygonal-support-uniform-bound}
  q_{p_\nu}(\theta)^2+
  q_{p_\nu}'(\theta)^2
  =
  |X_{p_\nu}(\theta)-p_\nu|^2
  \leqslant\operatorname{diam}(P)^2
\end{align}
on every open cell and for both one-sided derivatives at a~transition.  Thus the~functions $q_{p_\nu}$ are uniformly bounded and equi-Lipschitz.

Fix a~normal direction $\theta$ and let $L_\nu(\theta)$ be the~corresponding transformed line. Since $\Phi_P$ is involutive, its inverse image is a~secant through $p_\nu$ and, along a~subsequence, tends to a~line through $p_0$. If $p_0$ is interior, continuity of the~boundary intersections and of $\sigma_P$ makes $L_\nu(\theta)$ tend to the~unique line of $\Gamma^*_{P,p_0}$ with normal $\mathbbm{n}(\theta)$.

Suppose that $p_0\in\partial P$.  If the~limiting line meets $P$ in a~non-degenerate chord, then $p_0$ is one of its endpoints, except when the~line contains the~side whose relative interior contains $p_0$.  In the~first case the~limiting reflected line contains $\sigma_P(p_0)$.  In the~second case it is the~opposite supporting line, which also contains $\sigma_P(p_0)$.  If the~limiting chord degenerates to a~vertex, both reflected endpoints tend to the~opposite vertex.  Thus every subsequential limit of $L_\nu(\theta)$ is the~unique line with normal $\mathbbm{n}(\theta)$ passing through $p_0^*$.  Therefore
\begin{align*}
  q_{p_\nu}(\theta)\to
  (p_0^*-p_0)\cdot \mathbbm{n}(\theta).
\end{align*}
The~pointwise limit is unique and continuous. The~equi-Lipschitz bound \eqref{eq:polygonal-support-uniform-bound} therefore implies uniform convergence, whether $p_0$ is interior or on the~boundary.

We next consider the~derivatives.  If $p_0$ is interior, then outside the~finitely many transition directions the~two active sides are fixed for all sufficiently large $\nu$.  The~rational cell formula in Theorem~\ref{thm:exact-cellwise} then gives $q_{p_\nu}'(\theta)\to q_{p_0}'(\theta)$. If $p_0$ is on the~boundary, the~same conclusion holds outside a~finite set of limiting transition directions.  Indeed, on a~stable non-collapsing cell the~reflected lines converge in the~piecewise $C^1$ sense to a~pencil through $p_0^*$. On the~cell corresponding to the~whole supporting cone at a~vertex this follows from Proposition~\ref{prop:adjacent-pencil}, whose concurrence point tends to $p_0^*$.  Hence the~regular envelope points converge to $p_0^*$ and
\begin{align*}
  q_{p_\nu}'(\theta)
  =
  \mathbbm{t}(\theta)\cdot
  \bigl(X_{p_\nu}(\theta)-p_\nu\bigr)
  \to
  (p_0^*-p_0)\cdot\mathbbm{t}(\theta)
  =
  q_{p_0}'(\theta)
\end{align*}
for almost every $\theta$.  The~uniform bound \eqref{eq:polygonal-support-uniform-bound} and dominated convergence give convergence of the~derivatives in $L^2(0,\uppi)$.  Together with uniform convergence of the~functions, this proves \eqref{eq:polygonal-support-H1} and the~area-form assertion.
\end{proof}

\begin{theorem}
\label{thm:polygonal-art-area}
For every $\CPPOS$ $P$ and every $p\in\Int(P)$,
\begin{align}\label{eq:polygonal-art-area}
  A^*\bigl(\ART_P(p)\bigr)
  =
  \frac{1}{2}\sum_I
  \int_I
  \left(q_p(\theta)^2-q_p'(\theta)^2\right)\,\dd\theta
  \leqslant0,
\end{align}
where the~sum is over the~smooth cells of the~transformed line family. Equality in the~inequality holds if and only if the~completed $\ART$ degenerates to a~point.
\end{theorem}

\begin{proof}
The~function $q_p$ in \eqref{eq:polygonal-line-support} is continuous, anti-periodic, and $C^2$ on the~closures of its finitely many smooth cells, with one-sided derivatives at transitions. The~segment completion in Definition~\ref{def:discrete-art} agrees with that in the~piecewise support-area lemma of \cite{GJZARTSmooth}. That lemma gives \eqref{eq:polygonal-art-area}, its sign, and the~first-harmonic equality case.
\end{proof}

\begin{example}
\label{ex:no-area-ordering}
Consider the~$\CPPOS$ $P$ with vertices
\begin{align*}
\biggl(
\left(200,-\frac{100}{\sqrt{3}}\right),
\left(200,\frac{200}{\sqrt{3}}\right),
\left(0,\frac{400}{\sqrt{3}}\right),
\left(-200,\frac{200}{\sqrt{3}}\right),
\left(-200,-\frac{200}{\sqrt{3}}\right),
\left(-50,-\frac{350}{\sqrt{3}}\right)
\biggr).
\end{align*}
Let $c$ denote the~area centroid of $P$.  Direct polygonal calculation gives
\begin{align*}
  A^*(\CSS(P))
  =
  -\frac{160000}{63\sqrt{3}}
  \simeq-1466.286,
  \qquad
  A^*(\AreaEvolute(P))
  =
  -\frac{625}{\sqrt{3}}
  \simeq-360.844.
\end{align*}
Numerical evaluation of the~exact conic formula \eqref{eq:exact-cell-envelope}, including all transition segments, gives an~oriented ART area of approximately $-5033.780$ at the~centroid $c$. Its absolute value exceeds the~absolute areas of both the~CSS and the~Wigner caustic. At $p=(4P_0+c)/5$, the~oriented ART area is approximately $-340.631$, whose absolute value is smaller than both. Hence neither comparison with these two basepoint-independent areas holds universally with constant~$1$.
\end{example}

As the~basepoint approaches an~open side, the~oriented area tends to zero although the~completed front may converge to a~non-degenerate segment traversed in opposite directions.

\begin{theorem}
\label{thm:boundary-degeneration}
The~function $p\mapsto A^*\bigl(\ART_P(p)\bigr)$ extends continuously to $P$ by the~value $0$ on $\partial P$.

More precisely, let $x=P_i+se_i$ with $0<s<1$ and put
\begin{align}\label{eq:boundary-R}
 B_i(x)=P_{i+n}
 +\frac{\mu_{i-1}s}{\mu_{i-1}s+\mu_{i+1}(1-s)}e_{i+n}.
\end{align}
As $p\to x$ from $\Int(P)$, the~completed $\ART$ converges to the~segment $\operatorname{conv}\{\sigma_P(x),B_i(x)\}$ traversed once in each direction. It reduces to a~point precisely when $\mu_{i-1}=\mu_{i+1}$, in particular when $P$ is centrally symmetric.
\end{theorem}

\begin{proof}
By Theorem~\ref{thm:polygonal-art-area} and Lemma~\ref{lem:polygonal-support-continuity}, the~area depends continuously on $p\in\Int(P)$. As $p\to x\in\partial P$, the~support function converges in $H^1$ to the~first harmonic $(\sigma_P(x)-x)\cdot \mathbbm{n}(\theta)$, whose quadratic area form is zero. This proves continuity up to every side and vertex, independently of the~approach. We now identify the~completed-front limit for $x$ in an~open side.

For every chord direction separated from the~tangent direction at the~open side $e_i$, one endpoint tends to $x$.  Hence every corresponding reflected chord tends to a~line through $\sigma_P(x)$, and its envelope point tends to $\sigma_P(x)$.

It remains to resolve the~shrinking cell whose original chord has endpoints on the~sides $e_{i-1}$ and $e_{i+1}$. After translating $P_i$ to the~origin, let $\mathbbm{n}_i$ be the~inward unit normal to $e_i$ and write $p=s_\eps e_i+\eps \mathbbm{n}_i$, where $\eps=\operatorname{dist}(p,L_i)>0$ and $s_\eps\to s$. Write the~two endpoints as $-\eps Ue_{i-1}$ and $e_i+\eps Ve_{i+1}$, with $U,V\geqslant0$. Their limits are the~two ends of $e_i$. Since $s_\eps$ stays away from $0$ and $1$ and $p$ lies between the~endpoints, their transverse distances from $L_i$ are $O(\eps)$ uniformly on this cell. Hence $U$ and $V$ remain bounded. The~collinearity equation therefore gives
\begin{align}\label{eq:boundary-UV-relation}
 (1-s_\eps)U[e_{i-1},e_i]
 +s_\eps V[e_i,e_{i+1}]
 =[e_i,\mathbbm{n}_i]+O(\eps).
\end{align}
The~positive leading constant $[e_i,\mathbbm{n}_i]$ is independent of the~chord parameter.

The~reflected endpoints are $P_{i+n}+\eps\mu_{i-1}Ue_{i-1}$ and $P_{i+n+1}-\eps\mu_{i+1}Ve_{i+1}$. Seek their limiting envelope point in the~form $P_{i+n}+\rho e_{i+n}$. Expanding the~equation of the~joining line to first order in $\eps$ and eliminating $V$ with \eqref{eq:boundary-UV-relation}, the~coefficient of the~free parameter $U$ is
\begin{align*}
 [e_{i-1},e_i]
 \left((1-\rho)\mu_{i-1}
 -\rho\mu_{i+1}\frac{1-s_\eps}{s_\eps}\right).
\end{align*}
Consequently, the~envelope points tend uniformly to the~point with
\begin{align*}
 \rho
 =\lim_{\eps\to0^+}
 \frac{\mu_{i-1}s_\eps}
 {\mu_{i-1}s_\eps+\mu_{i+1}(1-s_\eps)}
 =\frac{\mu_{i-1}s}
 {\mu_{i-1}s+\mu_{i+1}(1-s)}.
\end{align*}
The~remaining constant in the~line equation affects only an~$O(\eps)$ transverse displacement. This proves \eqref{eq:boundary-R} for an~arbitrary approach to $x$.

The~two transition segments converge to the~same segment with opposite orientations, while every remaining conic arc collapses to $\sigma_P(x)$. This proves the~completed-front limit. The~limiting segment is traversed in opposite directions and therefore has zero oriented area.
\end{proof}

\begin{definition}
\label{def:art-invariant}
For $p\in P$ define
\begin{align*}
  \mathcal E_P(p)
  =
  \begin{cases}
  -A^*(\ART_P(p)),&p\in\Int(P),\\
  0,&p\in\partial P.
  \end{cases}
\end{align*}
The~\emph{normalised $\ART$ area invariant} and its maximising locus are
\begin{align}\label{eq:art-invariant}
  \delta_{\ART}(P)
  =
  \frac{\max_{p\in P}\mathcal E_P(p)}{A(P)},
  \qquad
  \mathcal C_{\ART}(P)
  =
  \operatorname*{arg\,max}_{p\in P}\mathcal E_P(p).
\end{align}
\end{definition}

\begin{proposition}
\label{prop:art-invariant-properties}
The~number $\delta_{\ART}(P)$ is invariant under nonsingular affine maps, and its non-empty compact maximising locus $\mathcal C_{\ART}(P)$ is affine covariant. Moreover, $\delta_{\ART}(P)=0$ if and only if $P$ is centrally symmetric. In that case $\mathcal E_P\equiv0$ and $\mathcal C_{\ART}(P)=P$. Otherwise $\mathcal C_{\ART}(P)\subset\Int(P)$.
\end{proposition}

\begin{proof}
Theorems~\ref{thm:polygonal-art-area} and \ref{thm:boundary-degeneration} show that $\mathcal E_P$ is a~continuous non-negative function on the~compact polygon $P$ and vanishes on its boundary.  This proves existence, compactness, and the~interior assertion.

Affine covariance of the~construction multiplies both areas in \eqref{eq:art-invariant} by the~same absolute determinant, proving the~invariance and covariance assertions.

If $P$ is centrally symmetric, the~final identities follow from Theorem~\ref{thm:discrete-properties}(e).  Conversely, suppose that $\delta_{\ART}(P)=0$.  If $P$ has four vertices, the~$\CPPOS$ condition makes it a~parallelogram, so it is already centrally symmetric.  We may therefore assume that $P$ has at least six vertices.

Non-negativity gives $\mathcal E_P(p)=0$ for every $p\in\Int(P)$.  By Theorem~\ref{thm:polygonal-art-area}, every $\Phi_P(\Lambda_p)$ is contained in a~projective pencil and, being an~embedded circle, equals that pencil. The~pencil-rigidity lemma of \cite{GJZARTSmooth} makes $\Phi_P$ the~restriction of a~projectivity of the~dual plane. Let $F$ be its primal projectivity.

Fix $x$ in the~relative interior of an~edge and choose two secants $xy_1,xy_2$ in distinct directions.  Since $F(xy_j)=\Phi_P(xy_j)= \overline{\sigma_P(x)\,\sigma_P(y_j)}$, the~two image lines meet at both $F(x)$ and $\sigma_P(x)$.  Hence $F(x)=\sigma_P(x)$.  Thus $F|_{\partial P}=\sigma_P$.

On the~edge $e_i$, the~restriction of $F$ is consequently $F(P_i+se_i)=P_{i+n}+se_{i+n}$. After projective completion this map fixes the~point at infinity of the~direction $e_i$.  Two non-parallel edge directions therefore show that $F$ preserves the~line at infinity, so $F$ is affine.  Write $F(x)=Bx+b$.  Since $Be_i=e_{i+n}=-\mu_i e_i$ and a~convex polygon with at least six vertices has at least three distinct edge directions, the~linear map $B$ has at least three eigenlines.  Hence $B=\lambda I$.  Since $\Phi_P^2=\operatorname{id}$ on the~open secant band, one has $F^2=\operatorname{id}$.  The~alternative $\lambda=1$ would make $F$ the~identity, whereas $\sigma_P$ has no fixed point.  Therefore $\lambda=-1$, and $F(x)=2o-x$ for $o=b/2$.  Since $F(P_i)=P_{i+n}$, the~polygon is centrally symmetric about $o$.
\end{proof}

\begin{remark}
\label{rem:involute-centre-comparison}
The~maximising locus $\mathcal C_{\ART}(P)$ is not, in general, the~central point obtained by iterating involutes.  Craizer proved that an~alternating sequence of involutes associated with the~Wigner caustic and $\CSS$ in a~Minkowski plane converges to a~unique point \cite[Theorem~5.1]{CraizerInvolutes}. Craizer and Martini established the~polygonal counterpart \cite[Theorem~5.1]{CraizerMartiniPolygons}.  Their point is unique even for centrally symmetric curves and polygons, whereas Proposition~\ref{prop:art-invariant-properties} gives the~entire polygon as the~ART-energy centre set in the~centrally symmetric case.

There is also no general containment of $\mathcal C_{\ART}(P)$ in an~``interior of the~Wigner caustic'', which need not be a~simple polygon. Even the~stronger unambiguous statement $\mathcal C_{\ART}(P) \subset\operatorname{conv}\bigl(\AreaEvolute(P)\bigr)$ is false: for a~centrally symmetric $\CPPOS$ the~left-hand side is $P$, whereas the~convex hull on the~right is its centre.

Numerical evidence shows that this failure is not confined to the~degenerate zero-energy case.  For example, reconstruct a~$\CPPOS$ from the~input vertices
\begin{align*}
 &(1.217503,-0.464101),\ (1.240617,0.485817),\\
 &(0.713513,1.039212),\ (-0.482154,1.068305),\\
 &(-1.151566,0.430696),\ (-1.167053,-0.205783),\\
 &(-0.238205,-1.180961),\ (0.447374,-1.197643).
\end{align*}
For the~reconstructed polygon described below, we obtain numerically that
\begin{align*}
 p_{\ART}\simeq(0.36896,0.23543),\qquad\delta_{\ART}(P)\simeq0.09612.
\end{align*}
To examine the~maximising locus, for $\eta>0$ consider the~near-maximising set
\begin{align}\label{eq:eta-near-maximiser}
  \mathcal C_{\ART}^{[\eta]}(P)
  =
  \left\{
  p\in P:
  \frac{\mathcal E_P(p)}{A(P)}
  \geqslant\delta_{\ART}(P)-\eta
  \right\}.
\end{align}
These nested sets intersect in $\mathcal C_{\ART}(P)$ as $\eta\to0^+$.
The~calculations used $0.09611757017622216$ for $\delta_{\ART}(P)$ in \eqref{eq:eta-near-maximiser}.  On each of $192$ equally spaced rays starting at $p_{\ART}$, the~energy was evaluated from $p_{\ART}$ to $\partial P$ on a~mixed logarithmic--linear mesh, and every first exit from the~level set was refined by $20$ bisection steps.  No ray re-entered any of the~seven tested level sets.  A~global energy grid and independent multi-start searches found no additional component.

Table~\ref{tab:art-near-maximising-sets} records the~estimated areas, diameters, and radial ranges for three of the~tested values of $\eta$.

\begin{table}[htbp]
\centering
\caption{Numerically estimated area, diameter, and radial range of the~near-maximising sets $\mathcal C_{\ART}^{[\eta]}(P)$}
\label{tab:art-near-maximising-sets}
\small
\begin{tabular}{@{}cccc@{}}
\toprule
$\eta$
&
$A\bigl(\mathcal C_{\ART}^{[\eta]}(P)\bigr)$
&
$\operatorname{diam}\bigl(\mathcal C_{\ART}^{[\eta]}(P)\bigr)$
&
radial range about $p_{\ART}$
\\
\midrule
$10^{-4}$ & $1.8236\cdot10^{-4}$ & $3.0326\cdot10^{-2}$
& $[6.4659\cdot10^{-4},\,2.9330\cdot10^{-2}]$\\
$10^{-5}$ & $2.6407\cdot10^{-6}$ & $4.1404\cdot10^{-3}$
& $[6.4681\cdot10^{-5},\,4.0408\cdot10^{-3}]$\\
$10^{-6}$ & $2.7974\cdot10^{-8}$ & $4.3764\cdot10^{-4}$
& $[6.4683\cdot10^{-6},\,4.2768\cdot10^{-4}]$\\
\bottomrule
\end{tabular}
\end{table}

At the~smallest tested scales the~diameter is approximately proportional to $\eta$ and the~area to $\eta^2$, suggesting a~strict, cone-like maximum at a~chamber intersection.

The~displayed coordinates are the~six-decimal input data used in the~original computation. For each pair of almost opposite outward unit normals $\mathbbm{n}_i,\mathbbm{n}_{i+n}$, the~common normal was chosen parallel to $\mathbbm{n}_i-\mathbbm{n}_{i+n}$, with opposite signs on the~two sides. Each support value was then evaluated at the~midpoint of its original side, and the~polygon was rebuilt from these support lines. The~maximum displacement from the~displayed vertices is less than $3.8\cdot10^{-7}$. This reconstructed $\CPPOS$ was used for the~numerical computation and Figure~\ref{fig:art-energy-maximizer}.
\end{remark}

\begin{figure}[ht]
  \centering
  \includegraphics[width=\linewidth]{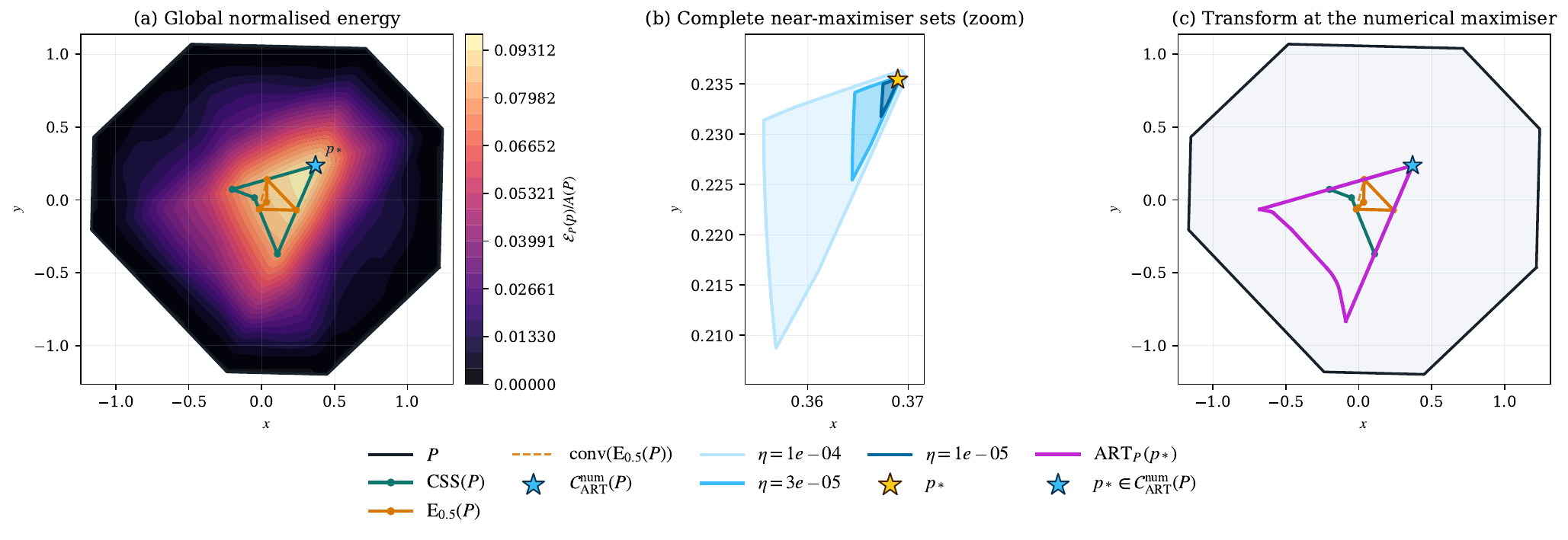}
\caption{The~explicit CPPOS. Left: the~global normalised energy. Centre: a~magnification of the~computed near-maximising sets for $\eta=10^{-4},3\cdot10^{-5},10^{-5}$. Right: the~completed $\ART$ at $p_{\ART}$}
  \label{fig:art-energy-maximizer}
\end{figure}

\begin{example}
\label{ex:two-art-energy-maximisers}
Let $P$ be the~$\CPPOS$ whose five upper vertices, listed from right to left, are
\begin{align*}
 \left(1,\tan\frac{\uppi}{10}\right),
 \quad
 &\left(\frac{10\sqrt{5}-13}{20},\frac{43+3\sqrt{5}}{20\sqrt{10-2\sqrt{5}}}\right),
 \quad \left(\frac{3(1+\sqrt{5})}{40},\frac{43(\sqrt{5}-1)}{20\sqrt{10-2\sqrt{5}}}\right),
 \\
 &
 \left(\frac{17-23\sqrt{5}}{40},\frac{43-3\sqrt{5}}{20\sqrt{10-2\sqrt{5}}}\right),
 \quad
 \left(-\frac{6}{5},\frac{17-6\sqrt{5}}{5\sqrt{10-2\sqrt{5}}}\right).
\end{align*}
The~remaining five vertices are their reflections in the~$x$-axis, taken in reverse order. Geometrically, $P$ is an~equiangular decagon with parallel opposite sides and all interior angles equal to $4\uppi/5$. It is symmetric about the~$x$-axis but not centrally symmetric.

Numerical integration of the~exact cellwise formulas, combined with a~global grid search and independent multi-start searches, gives two symmetric numerical maximisers,
\begin{align*}
  p_\pm
  \simeq
  (0.24271,\ \pm0.50679),
  \qquad
  \frac{\mathcal E_P(p_\pm)}{A(P)}
  \simeq0.25296.
\end{align*}
Reflection symmetry forces the~two values to agree. No further component of the~top level sets was resolved (Figure~\ref{fig:art-energy-nonunique-examples}).

Both points lie in the~relative interior of the~same edge of $\CSS(P)$.  Their distance from the~nearest $\CSS$ vertex is approximately $0.24018$, while their distance from $\AreaEvolute(P)$ is approximately $0.48199$.  Thus neither candidate is an~edge-cusp vertex of either front. The~computation suggests a~finite disconnected centre set.
\end{example}

\begin{example}
\label{ex:art-energy-maximiser-off-fronts}
Let $\widetilde P$ be the~$\CPPOS$ with vertices
\begin{align*}
 \biggl(&
 \left(1,-\frac{1}{\sqrt{3}}\right),
 \left(1,\frac{6}{5\sqrt{3}}\right),
 \left(\frac{1}{5},\frac{2}{\sqrt{3}}\right),
 \left(-\frac{6}{5},\frac{3}{5\sqrt{3}}\right),
 \left(-\frac{6}{5},-\frac{2}{5\sqrt{3}}\right),
 \left(\frac{1}{5},-\frac{9}{5\sqrt{3}}\right)
 \biggr).
\end{align*}
Geometrically, $\widetilde P$ is an~equiangular hexagon with interior angles $2\uppi/3$. Its area is
\begin{align*}
  A(\widetilde P)=\frac{144}{25\sqrt{3}}.
\end{align*}

The~numerical maximiser is
\begin{align*}
  \widetilde p_{\ART}
  \simeq
  (-0.22566,0.05774),
  \qquad
  \delta_{\ART}(\widetilde P)
  \simeq0.17388.
\end{align*}
The~eigenvalues of the~finite-difference Hessian of the~normalised energy at this point are approximately $-0.09392$ and $-0.07737$, which is consistent with a~smooth strict local maximum but does not prove it.  Moreover,
\begin{align*}
  \operatorname{dist}
  \bigl(\widetilde p_{\ART},\CSS(\widetilde P)\bigr)
  \simeq0.11105,
  \qquad
  \operatorname{dist}
  \bigl(\widetilde p_{\ART},\AreaEvolute(\widetilde P)\bigr)
  \simeq0.12566.
\end{align*}
Thus the~candidate lies on neither front. The~global grid and multi-start searches found no competing maximum at the~stated resolution (Figure~\ref{fig:art-energy-nonunique-examples}).
\end{example}

\begin{figure}[ht]
  \centering
  \includegraphics[width=\linewidth]{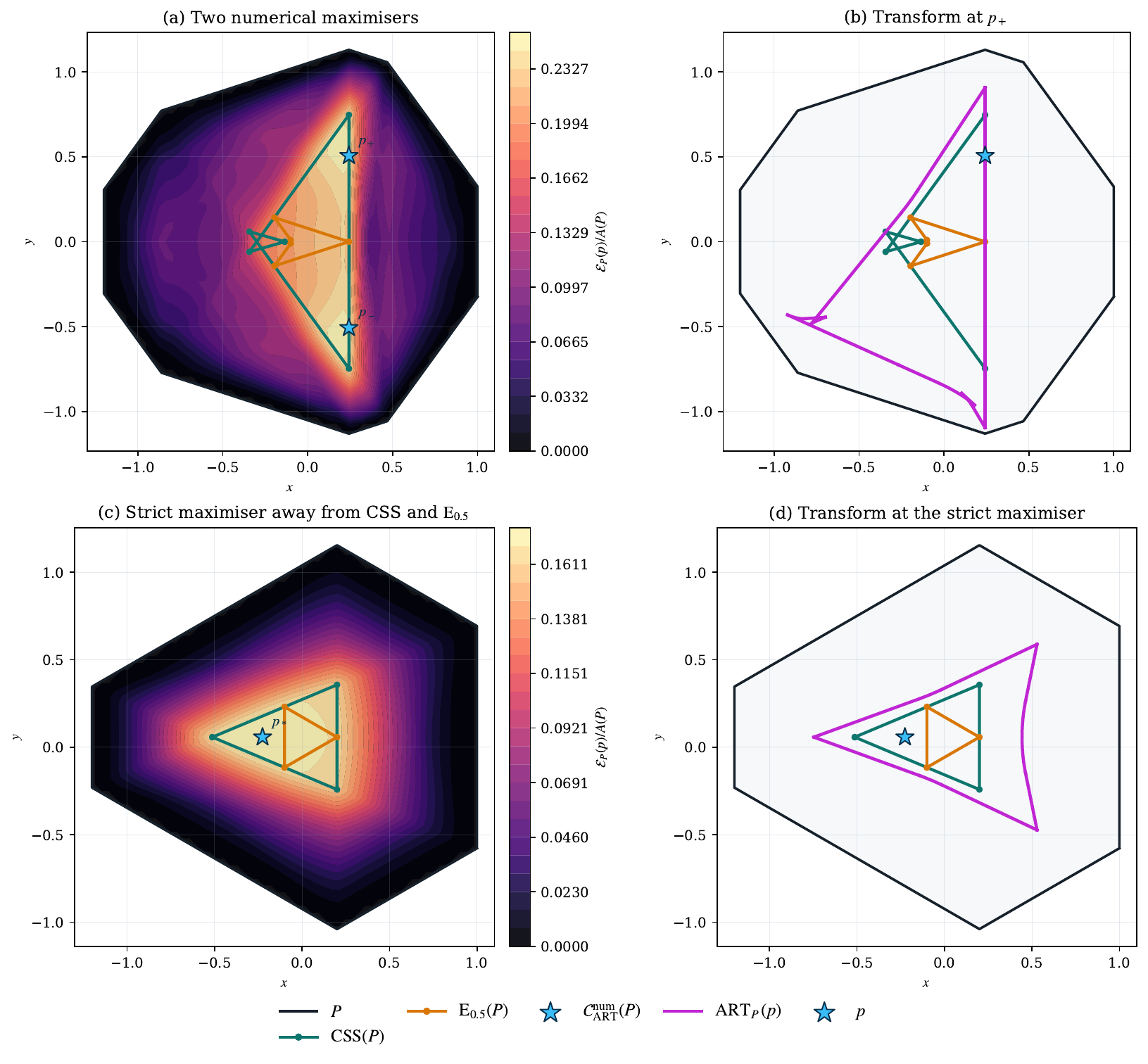}
\caption{Two numerically observed behaviours of the~ART-energy centre set.  Top: the~axis-symmetric decagon of Example~\ref{ex:two-art-energy-maximisers}, with two symmetric numerical maximisers, and the~completed $\ART$ at $p_+$.  Bottom: the~hexagon of Example~\ref{ex:art-energy-maximiser-off-fronts}, whose computed maximiser candidate lies on neither the~CSS nor the~Wigner caustic, together with its completed ART}
  \label{fig:art-energy-nonunique-examples}
\end{figure}

\begin{conjecture}
\label{conj:no-noncentral-art-plateau}
If a~$\CPPOS$ $P$ is not centrally symmetric, then $\Int\bigl(\mathcal C_{\ART}(P)\bigr)=\varnothing$.
\end{conjecture}

No non-centrally symmetric counterexample to Conjecture~\ref{conj:no-noncentral-art-plateau} was found in the~numerical families tested here.

\begin{conjecture}
\label{conj:ART-CSS-area}
For every $\CPPOS$ $P$ and every $p\in\Int(P)$, $\mathcal E_P(p)\leqslant 4\left|A^*\bigl(\CSS(P)\bigr)\right|$. The~constant $4$ is sharp in a~limiting sense.
\end{conjecture}

Tests using the~exact cellwise formulas found no violation in the~families considered. Targeted families approaching central symmetry produced ratios as high as $3.99844$. Proposition~\ref{prop:css-bound-limiting-sharpness} below proves the~limiting sharpness of the~constant $4$.

\begin{conjecture}
\label{conj:art-css-hull}
For every non-centrally symmetric $\CPPOS$ $P$, $\mathcal C_{\ART}(P)\subseteq\operatorname{conv}\bigl(\CSS(P)\bigr)$.
\end{conjecture}

\section{Compatible smoothing and convergence to the~smooth transform}
\label{sec:polygon-smooth-approximation}

\noindent An~\emph{oval} is a~smooth strictly convex simple closed curve with positive curvature. Its bounded interior is denoted by $\Int(C)$. Its parallel-tangent involution $\iota_C$ sends each point to the~other point with parallel tangent. The~smooth reflected line map $\Phi_C$ applies $\iota_C$ to both chord endpoints, and $\ART_C(p)$ is the~envelope of the~image of an~interior pencil. We use the~regularity, strict direction monotonicity, and containment properties established in \cite{GJZARTSmooth}.

\Needspace{7\baselineskip}
\subsection{Compatible smoothing and chord counts}

\begin{lemma}\label{lem:compatible-smoothing}
Every $\CPPOS$ $P$ admits smooth strictly convex ovals $C_\eps$ and orientation-preserving homeomorphisms $h_\eps:\partial P\to C_\eps$ such that
\begin{align}\label{eq:compatible-smoothing}
 h_\eps\to\operatorname{id}_{\partial P}\quad\text{uniformly},
 \qquad
 \iota_{C_\eps}\circ h_\eps=h_\eps\circ\sigma_P.
\end{align}
On compact subsegments of open sides, the~convergence is of first order in the~affine edge parameter.
\end{lemma}

\begin{proof}
Let $\alpha_i$ be the~outward normal angle of $e_i$ and put $\ell_i=|e_i|$. Choose an~even $\eta\in C_c^\infty(\R)$, supported on $[-1,1]$, positive on $(-1,1)$, and of integral~$1$. For small $\eps>0$, let $\eta_\eps(t)=\eps^{-1}\eta(t/\eps)$, extended periodically, and set
\begin{align*}
 \rho_\eps(\theta)=\sum_i\ell_i\eta_\eps(\theta-\alpha_i)+\eps^2.
\end{align*}
Choose $\eps$ so that the~kernel supports are disjoint. Since $\sum_i\ell_i\mathbbm{t}(\alpha_i)=\sum_i e_i=0$, evenness of the~kernel gives $\int_0^{2\uppi}\rho_\eps\mathbbm{t}=0$. Thus $\gamma_\eps(\theta)=q_\eps+\int_0^\theta\rho_\eps(u)\mathbbm{t}(u)\,\dd u$ parametrises a~closed smooth oval $C_\eps$. Fix the~translation by $\gamma_\eps(\alpha_0-\eps)=P_0$.

To identify its boundary with $\partial P$, choose an~increasing homeomorphism $H_\eps:\R/(2n\mathbb Z)\to\R/(2\uppi\mathbb Z)$. On $\sqrt\eps\leqslant s\leqslant1-\sqrt\eps$, prescribe it by
\begin{align}\label{eq:smoothing-cumulative}
 \int_{\alpha_i-\eps}^{H_\eps(i+s)}\eta_\eps(u-\alpha_i)\,\dd u
 =\frac{s-\sqrt\eps}{1-2\sqrt\eps}.
\end{align}
Across each remaining vertex interval extend it monotonically through the~gap between successive kernel supports. Because $\alpha_{i+n}=\alpha_i+\uppi$, the~extensions can be chosen to satisfy
\begin{align}\label{eq:smoothing-antipodal-H}
 H_\eps(r+n)=H_\eps(r)+\uppi.
\end{align}
Define
\begin{align}\label{eq:smoothing-homeomorphism}
 h_\eps(P_i+se_i)=\gamma_\eps(H_\eps(i+s)).
\end{align}
The~definitions agree at vertices and give an~orientation-preserving homeomorphism.

Successive integration gives $\gamma_\eps(\alpha_i-\eps)=P_i+O(\eps)$ uniformly in~$i$. On the~central edge interval, \eqref{eq:smoothing-cumulative} and $\mathbbm{t}(u)=\mathbbm{t}(\alpha_i)+O(\eps)$ give
\begin{align*}
 h_\eps(P_i+se_i)
 =P_i+\frac{s-\sqrt\eps}{1-2\sqrt\eps}e_i+O(\eps)
 =P_i+se_i+O(\sqrt\eps).
\end{align*}
On each omitted vertex interval both the~polygonal point and its image are $O(\sqrt\eps)$ from the~same vertex. The~motion through the~normal gap caused by the~term $\eps^2$ is $O(\eps^2)$. Hence $h_\eps\to\operatorname{id}$ uniformly. On a~fixed compact subinterval of $0<s<1$, the~kernel value $\eta_\eps(H_\eps(i+s)-\alpha_i)$ is bounded below by a~positive multiple of $\eps^{-1}$. Differentiating \eqref{eq:smoothing-cumulative} and \eqref{eq:smoothing-homeomorphism} therefore gives
\begin{align*}
 \frac{\dd}{\dd s}h_\eps(P_i+se_i)
 =\frac{\ell_i}{1-2\sqrt\eps}\mathbbm{t}(H_\eps(i+s))+O(\eps^3)
 =e_i+O(\sqrt\eps).
\end{align*}
Finally, $\iota_{C_\eps}(\gamma_\eps(\theta))=\gamma_\eps(\theta+\uppi)$ and \eqref{eq:smoothing-antipodal-H} give the~exact identity in \eqref{eq:compatible-smoothing}.
\end{proof}

\begin{proposition}
\label{prop:discrete-chambers}
Suppose that $p\in\Int(P)$ is generic: all chords with midpoint $p$ are isolated and have endpoints in open non-parallel sides, and all fixed chords through $p$ have endpoints in open opposite sides and are transverse solutions of the~incidence equation. Such points form an~open dense subset of $\Int(P)$. If $M_P(p)$ counts chords with midpoint $p$ and $T_P(p)$ counts fixed chords of $\Phi_P$ through $p$, then
\begin{align}\label{eq:discrete-count-bound}
 M_P(p)\leqslant T_P(p).
\end{align}
Each fixed chord is a~common generating line of $\ART_P(p)$ and $\CSS(P)$, so the~two completed duals have at least $M_P(p)$ common generating lines.
\end{proposition}

\begin{proof}
For an~unoriented direction $\alpha$, let $x_\alpha,y_\alpha$ be the~two boundary intersections of the~line through $p$, chosen continuously over one projective period. The~incidence functions
\begin{align*}
 F_P(\alpha)=|x_\alpha-p|-|y_\alpha-p|,\quad 
 G_P(\alpha)=[x_\alpha-p,\sigma_P(x_\alpha)-p]
\end{align*}
have well-defined zero sets on the~direction circle, even if their chosen lifts change sign after one period. Their zeros are precisely the~centred and fixed chords, respectively: in the~second case $\sigma_P(x_\alpha)$ is the~other boundary intersection because $p$ is interior and $\sigma_P$ is fixed-point-free.

Exclude the~finite union of midpoint segments
\begin{align*}
 \bigcup_{i,j}\left\{\frac{P_i+y}{2}:y\in\operatorname{conv}\{P_j,P_{j+1}\}\right\},
\end{align*}
the~flat midpoint segments of Remark~\ref{rem:flat-discriminant}, and the~great-diagonal lines $d_i$. This is a~closed set with empty interior. Outside it, all relevant endpoints lie in open sides. A~multiple zero of $F_P$ would give $x_\alpha'+y_\alpha'=0$, forcing the~two active sides to be parallel. For a~fixed chord on $e_i$, the~joining lines pass through $D_i$. Their incidence determinant is affine in the~side parameter and has a~simple zero unless $p=D_i$, already excluded. Thus the~generic conditions hold on a~dense set. They are open because simple zeros persist and compactness of the~direction circle excludes additional zeros away from them.

Apply Lemma~\ref{lem:compatible-smoothing} and define $F_\eps,G_\eps$ from $C_\eps$ in the~same way. They converge uniformly to $F_P,G_P$ and in $C^1$ on closed intervals where the~limiting endpoints lie in open sides. The~latter follows from first-order boundary convergence and transversality of chords through an~interior point. Each simple zero therefore persists uniquely, while on the~compact complement the~limiting incidence functions are bounded away from zero. Consequently, for sufficiently small $\eps$,
\begin{align*}
 M_{C_\eps}(p)=M_P(p),\qquad T_{C_\eps}(p)=T_P(p).
\end{align*}
The~smooth chord-count inequality \cite{GJZARTSmooth} now proves \eqref{eq:discrete-count-bound}. The~common-line statement follows from Theorem~\ref{thm:discrete-properties}(b) and Definition~\ref{def:discrete-art}.
\end{proof}

\begin{remark}\label{rem:flat-discriminant}
For a~smooth oval, the~Wigner caustic is the~discriminant of the~centred-chord equation. For a~polygon, each opposite pair of sides gives a~midpoint segment
\begin{align*}
 \left\{\frac{x+y}{2}:
 x\in\operatorname{conv}\{P_i,P_{i+1}\},\ 
 y\in\operatorname{conv}\{P_{i+n},P_{i+n+1}\}\right\}.
\end{align*}
It contains the~canonical Wigner caustic edge $\operatorname{conv}\{M_i,M_{i+1}\}$, and every point in its relative interior is the~midpoint of a~continuum of chords on that side pair. Their union can therefore exceed $\AreaEvolute(P)$. Thus the~canonical midpoint locus is not the~full discriminant of arbitrary centred chords. Proposition~\ref{prop:discrete-chambers} asserts stability only for generic basepoints.
\end{remark}

\subsection{Hausdorff convergence}

\noindent Let $C$ be a~$C^4$ oval parametrised by outward normal angle, $\gamma:\R/2\uppi\mathbb Z\to C$. For finite direction sets $\nu_k\subset[0,\uppi)$, let $P_k=\operatorname{PPOS}_{\nu_k}(C)$ be the~circumscribed $\CPPOS$ formed from the~tangents with normals in $\nu_k\cup(\nu_k+\uppi)$, as in \cite{CraizerTeixeiraSilva,KonicerEtAl}. Write $\delta_k=\mesh(\nu_k)$ for the~largest consecutive gap in this paired set.

\begin{theorem}\label{thm:hausdorff}
If $\delta_k\to0$, then, for every $p\in\Int(C)$,
\begin{align*}
 d_H\bigl(\ART_{P_k}(p),\ART_C(p)\bigr)\to0.
\end{align*}
\end{theorem}

\begin{proof}
Let $t_i$ be successive sampled angles, put $\Delta_i=t_{i+1}-t_i$, and let $\rho$ be the~radius of curvature of $C$. Relative to $\gamma(t_i)$, the~two vertices on its tangent have tangent coordinates
\begin{align*}
 -\frac{1}{2}\rho(t_i)\Delta_{i-1}+O(\Delta_{i-1}^2),\qquad
 \frac{1}{2}\rho(t_i)\Delta_i+O(\Delta_i^2).
\end{align*}
Parametrise that side affinely over $[t_i-\Delta_{i-1}/2,t_i+\Delta_i/2]$. Taylor expansion gives continuous piecewise $C^1$ parametrisations $r_k$ satisfying
\begin{align}
 \sigma_{P_k}(r_k(t))=r_k(t+\uppi),
\quad r_k(t)=\gamma(t)+O(\delta_k^2),\quad\text{and}\quad
 r_k'(t)=\gamma'(t)+O(\delta_k)\label{eq:paired-c1}
\end{align}
uniformly, including both one-sided derivatives. The~first identity holds because opposite sides receive the~same affine parameter. The~derivative error is bounded by
\begin{align*}
 O\left(\frac{\Delta_{i-1}^2+\Delta_i^2}{\Delta_{i-1}+\Delta_i}\right)=O(\delta_k),
\end{align*}
which explains why no quasi-uniformity assumption is needed.

Let $u(t)$ and $u_k(t)$ be the~other endpoint parameters of the~chords through $p$ and $\gamma(t)$ or $r_k(t)$. They solve the~non-diagonal branches of $[\gamma(t)-p,\gamma(u)-p]=0$ and $[r_k(t)-p,r_k(u_k)-p]=0$. Transversality gives a~uniform positive lower bound for $|[\gamma(t)-p,\gamma'(u(t))]|$. The~implicit function theorem and \eqref{eq:paired-c1} therefore give $u_k\to u$ together with both one-sided first derivatives. Hence the~homogeneous reflected line coefficients
\begin{align*}
 \boldsymbol\ell_k(t)
 =\widehat{r_k(t+\uppi)}\times\widehat{r_k(u_k(t)+\uppi)}
 \to
 \widehat{\gamma(t+\uppi)}\times\widehat{\gamma(u(t)+\uppi)}
 =\boldsymbol\ell(t)
\end{align*}
converge uniformly in the~piecewise $C^1$ sense.

The~smooth dual curve is a~regular embedding, so, after continuous normalisation of its lift, $\boldsymbol\ell\times\boldsymbol\ell'$ is uniformly non-zero. Formula~\eqref{eq:dual-envelope} gives uniform convergence of all regular cellwise envelope points to the~smooth envelope. At a~cell boundary the~two one-sided envelope points converge to the~same limit, so each completing segment has diameter $O(\delta_k)$. Every smooth parameter value is represented by a~cell point or a~transition. Thus both Hausdorff inclusions hold.
\end{proof}

\section{Extremal bounds for the~normalised ART energy}
\label{sec:extremal-art-energy}

\noindent We bound $\delta_{\ART}(P)$ using the~area of the~difference body $P-P=\{x-y:x,y\in P\}$. An~analogous sharp estimate is proved for smooth ovals in the~companion preprint \cite{GJZARTSmooth}.

\begin{theorem}\label{thm:sharp-art-energy-bound}\label{conj:sharp-normalised-art-energy}
For every $\CPPOS$ $P$ and every $p\in\Int(P)$,
\begin{align}\label{eq:sharp-art-difference-bound}
 \mathcal E_P(p)
 <\frac{A(P-P)-2A(P)}{4}
 \leqslant A(P).
\end{align}
Consequently, $\delta_{\ART}(P)<1$. The~constant $1$ is optimal. Precisely,
\begin{align*}
 \sup_{P\ {\rm CPPOS}}\delta_{\ART}(P)=1,
\end{align*}
and no polygon attains this supremum.
\end{theorem}

\begin{proof}
Parametrise the~oriented chords through $p$ by $t\in\R/2\uppi\mathbb Z$. Let $a(t)$ and $b(t)$ be the~reflections of their positive and negative endpoints, respectively, and put $v=b-a$. Both $a$ and $b$ traverse $\partial P$ once positively, and $a(t+\uppi)=b(t)$, $b(t+\uppi)=a(t)$. These maps are continuous and piecewise smooth. On each open cell the~endpoints move positively along two distinct open sides. Convexity therefore gives $[a',v]>0$ and $[v,b']>0$, so
\begin{align*}
 [v,v']=[a',v]+[v,b']>0.
\end{align*}
Differentiating the~line equation $[v,z-a]=0$ shows that its envelope point is
\begin{align*}
 X=a+\lambda v,
 \qquad
 \lambda=\frac{[a',v]}{[v,v']}.
\end{align*}

The~direction of $v$ makes one positive turn by Proposition~\ref{prop:polygonal-direction}. Hence $v$ bounds a~star-shaped region $\Omega$ about the~origin, with
\begin{align*}
 A(\Omega)=\frac{1}{2}\int_0^{2\uppi}[v,v']\,\dd t.
\end{align*}
Since $v(t)\in P-P$ and $P-P$ is convex and contains the~origin, $\Omega\subset P-P$.

A~cellwise calculation gives
\begin{align*}
 [X,X']
 =[a,a']+\bigl(\lambda[a,v]\bigr)'
 -\frac{[a',v]^2}{[v,v']}.
\end{align*}
At a~transition, $a$ and $v$ are continuous and the~completing segment joins $a+\lambda_-v$ to $a+\lambda_+v$. Its contribution to $\int[X,\dd X]$ is $(\lambda_+-\lambda_-)[a,v]$. This cancels the~corresponding jump in the~cellwise integral of the~total derivative. Over $0\leqslant t\leqslant2\uppi$, the~completed front is traversed twice, whereas $a$ traverses $\partial P$ once. Therefore
\begin{align*}
 4A^*(\ART_P(p))
 =2A(P)-\int_0^{2\uppi}\frac{[a',v]^2}{[v,v']}\,\dd t.
\end{align*}
The~shift $t\mapsto t+\uppi$ interchanges $[a',v]$ and $[v,b']$ and preserves $[v,v']$. Averaging the~two integral expressions consequently yields
\begin{align}\label{eq:art-chord-difference-area}
 4A^*(\ART_P(p))
 &=2A(P)-A(\Omega)
 +\int_0^{2\uppi}\frac{[a',v][v,b']}{[v,v']}\,\dd t.
\end{align}
The~last integral in \eqref{eq:art-chord-difference-area} is strictly positive. Thus $A(\Omega)\leqslant A(P-P)$ gives the~first inequality in \eqref{eq:sharp-art-difference-bound}. The~planar Rogers--Shephard inequality $A(P-P)\leqslant6A(P)$ \cite{RogersShephard} gives the~second. Finally, $\mathcal E_P$ attains its maximum on $P$ and vanishes on $\partial P$. Applying the~strict bound at an~interior maximiser (with the~zero-maximum case immediate) gives $\delta_{\ART}(P)<1$. Example~\ref{prop:unit-energy-cppos-family} proves optimality.
\end{proof}

\begin{example}
\label{prop:unit-energy-cppos-family}
For $0<\eps<1/6$, let $P_\varepsilon$ be a~$\CPPOS$ with the~following vertices:
\begin{align*}
 &\left(\frac{1}{3}+\eps,\sqrt{3}\left(\eps-\frac{1}{3}\right)\right),
 \left(\frac{1}{3}+\eps,\sqrt{3}\left(\frac{1}{3}-\eps\right)\right),
 \left(\frac{1}{3}-2\eps,\frac{\sqrt{3}}{3}\right),
 \left(\eps-\frac{2}{3},\sqrt{3}\eps\right),\\
 &\left(\eps-\frac{2}{3},-\sqrt{3}\eps\right),
 \left(\frac{1}{3}-2\eps,-\frac{\sqrt{3}}{3}\right).
\end{align*}
Geometrically, $P_\varepsilon$ is obtained by symmetrically truncating the~three vertices of an~equilateral triangle by lines parallel to the~opposite sides.  Its side lengths alternate between $2\sqrt{3}(1/3-\varepsilon)$ and $2\sqrt{3}\varepsilon$.

Thus a~direct area calculation gives
\begin{align}\label{eq:unit-energy-family-area}
  A(P_\eps)
  =
  \sqrt{3}
  \left(
    \frac{1}{3}+2\eps-6\eps^2
  \right)
  \to
  \frac{1}{\sqrt{3}}.
\end{align}

At the~basepoint $p=0$, the~six cells of the~reflected-line family split, by the~threefold rotational symmetry, into three congruent non-degenerate cells and three degenerate cells.  We record the~limit calculation for one non-degenerate cell.  Take the~original endpoints on $e_4$ and $e_0$, and use the~affine parameter $s$ on $e_4$.  Its active interval tends to $[0,1/2]$, and the~corresponding parameter on $e_0$ tends to $(1-s)/(2-3s)$. Put $d=2-3s$.  Substitution in the~exact formula of Theorem~\ref{thm:exact-cellwise}, followed by cancellation of the~common first-order factor in $\eps$, gives the~limiting envelope
\begin{align*}
  Z_0(d)
  =
  \left(
    \frac{1-2d^2}{3(1+d^2)},
    \frac{1}{\sqrt{3}(1+d^2)}
  \right),
\end{align*}
traversed from $d=2$ to $d=1/2$. After cancellation, the~rational parametrisation and its first derivative converge uniformly on the~closed active interval: their limiting denominator factors $2-3s$ and $1+(2-3s)^2$ stay away from zero there. Thus the~arc-area integrals converge as well. Since
\begin{align*}
  [Z_0(d),Z_0'(d)]
  =
  \frac{4d}{3\sqrt{3}(1+d^2)^2},
\end{align*}
the~oriented area of this arc tends to
\begin{align*}
  \frac{1}{2}\int_2^{1/2}[Z_0(d),Z_0'(d)]\,\dd d
  =-\frac{\sqrt{3}}{15}.
\end{align*}

The~three opposite-side cells degenerate to points, so their arc areas are $o(1)$.  At the~two ends of the~representative arc, the~adjacent degenerate-cell points tend to
\begin{align*}
  Q_0=\left(-\frac{2}{3},0\right),
  \qquad
  Q_1=\left(\frac{1}{3},\frac{\sqrt{3}}{3}\right).
\end{align*}
Moreover,
\begin{align*}
  Z_0(2)=\left(-\frac{7}{15},\frac{\sqrt{3}}{15}\right),
  \qquad
  Z_0\left(\frac{1}{2}\right)
  =\left(\frac{2}{15},\frac{4\sqrt{3}}{15}\right),
\end{align*}
and the~two oriented transition segments satisfy
\begin{align*}
  \frac{1}{2}[Q_0,Z_0(2)]
  =
  \frac{1}{2}\left[Z_0\left(\frac{1}{2}\right),Q_1\right]
  =-\frac{\sqrt{3}}{45}.
\end{align*}
By threefold rotational symmetry, the~three non-degenerate arcs have equal limiting area contributions, as do the~six transition segments. Consequently,
\begin{align*}
 A^*\bigl(\ART_{P_\eps}(0)\bigr)
 =
 -3\frac{\sqrt{3}}{15}
 -6\frac{\sqrt{3}}{45}
 +o(1)
 =
 -\frac{1}{\sqrt{3}}+o(1).
\end{align*}
Together with \eqref{eq:unit-energy-family-area}, this gives
\begin{align*}
\lim_{\varepsilon\to 0^+}\frac{A^*\bigl(\ART_{P_\eps}(0)\bigr)}{A(P_\varepsilon)}=-1.
\end{align*}
Since $\delta_{\ART}(P_\varepsilon)\geqslant\mathcal E_{P_\varepsilon}(0)/A(P_\varepsilon)$, it follows that
\begin{align*}
  \lim_{\varepsilon\to0^+}\delta_{\ART}(P_\varepsilon)=1,
\end{align*}
where the~upper bound follows from Theorem~\ref{thm:sharp-art-energy-bound}.
\end{example}

\begin{proposition}
\label{prop:css-bound-limiting-sharpness}
For the~family $P_\eps$ of Example~\ref{prop:unit-energy-cppos-family},
\begin{align}\label{eq:css-bound-limiting-sharpness}
 \lim_{\eps\to(1/6)^-}
 \frac{\mathcal E_{P_\eps}(0)}{|A^*(\CSS(P_\eps))|}=4.
\end{align}
Consequently, any universal inequality of the~form $\mathcal E_P(p)\leqslant c|A^*(\CSS(P))|$ must have $c\geqslant4$.
\end{proposition}

\begin{proof}
Write $h=\eps-1/6<0$. At $h=0$ the~vertices in Example~\ref{prop:unit-energy-cppos-family} form a~regular hexagon centred at the~origin. Direct substitution in the~formula for $D_i$ gives
\begin{align*}
 D_0=(4h,0),\qquad D_1=(-2h,-2\sqrt{3}h),\qquad D_2=(-2h,2\sqrt{3}h),
\end{align*}
and therefore $A^*(\CSS(P_\eps))=-12\sqrt{3}h^2$. For the~opposite-side cell supported on $e_0$ and $e_3$, we have $\mu_0=(1+6h)/(1-6h)$ and $\tau_0(0)=1/2+h$. Formula~\eqref{eq:Qi-mu} consequently yields
\begin{align*}
 Q_0(0)=\left(\frac{8h(1+12h^2)}{1+60h^2},0\right).
\end{align*}
The~other two opposite-side cell points are its rotations through angles $2\uppi/3$ and $4\uppi/3$.

Let $q_h$ be the~support function of the~reflected line family relative to the~origin. We first show that $q_h$ and both one-sided derivatives $q_h'$ are uniformly $O(|h|)$. Parametrise the~original chords by their direction angle $\alpha$, with endpoints $x_h(\alpha)$ and $y_h(\alpha)$. All sufficiently small $|h|$ give polygons containing a~fixed disc about the~origin and lying in a~fixed disc. Hence the~endpoints and their one-sided derivatives are uniformly bounded, and the~chord lengths are bounded below. The~affine midpoint map on each side is $O(h)$, while every side ratio is $1+O(h)$. Thus, writing $a_h=\sigma_{P_\eps}(x_h)$ and $b_h=\sigma_{P_\eps}(y_h)$, we obtain
\begin{align*}
 a_h+x_h&=O(h), & a_h'+x_h'&=O(h),\\
 b_h+y_h&=O(h), & b_h'+y_h'&=O(h),
\end{align*}
uniformly on all cells, where primes in these estimates denote derivatives with respect to $\alpha$. Since $[x_h,y_h]=0$, the~constant coefficient $[a_h,b_h]$ of the~reflected line and its derivative are $O(h)$. Moreover, $b_h-a_h=-(y_h-x_h)+O(h)$ in the~piecewise $C^1$ sense, so its normal angle satisfies $\dd\theta/\dd\alpha=1+O(h)$. Normalising the~line equation proves the~claimed bounds for $q_h$ and its derivative with respect to $\theta$.

At $h=0$, the~cell supported on $e_0$ and $e_3$ has normal-angle interval $[\uppi/3,2\uppi/3]$, modulo~$\uppi$. The~other two opposite-side cells are its rotations. The~vertex directions vary smoothly with $h$, so the~three remaining cells have total normal-angle length $O(|h|)$. Their contribution to the~area integral in Theorem~\ref{thm:polygonal-art-area} is therefore $O(|h|^3)$. On the~first opposite-side cell, $q_h(\theta)=Q_0(0)\cdot \mathbbm{n}(\theta)$. Dividing its area contribution by $h^2$ and taking the~limit gives
\begin{align*}
 \frac{1}{2}\int_{\uppi/3}^{2\uppi/3}64\cos(2\theta)\,\dd\theta
 =-16\sqrt{3}.
\end{align*}
Threefold rotational symmetry now gives $\mathcal E_{P_\eps}(0)=48\sqrt{3}h^2+o(h^2)$. Comparison with the~exact CSS area proves~\eqref{eq:css-bound-limiting-sharpness}.
\end{proof}

\section*{Declarations}

\noindent\textbf{Funding.} The~authors received no financial support for this work.

\noindent\textbf{Conflict of interest.} The~authors declare that they have no competing interests.

\noindent\textbf{Data availability.} No external datasets were used. 

\noindent\textbf{Software.} The~authors used Python~\cite{Python} and Mathematica~\cite{WolframMathematica} for the~figures and selected computations.

\noindent\textbf{Use of generative artificial intelligence.} The~authors used OpenAI's ChatGPT to assist with literature searches, language editing, and checking mathematical derivations and computations. The~authors take full responsibility for the~final content.

\end{document}